\documentclass{amsart}

\usepackage{amssymb, tikz}

\usetikzlibrary{shapes,decorations.pathmorphing,backgrounds,positioning,fit,petri, hobby}
\usepackage{enumitem}
\usepackage[colorlinks=true,linkcolor=red!50!black,citecolor=green!50!black]{hyperref}

\usepackage{multicol}

\newtheorem{theorem}{Theorem}
\numberwithin{theorem}{subsection}
\newtheorem{lemma}[theorem]{Lemma}

\newtheorem{prop}[theorem]{Proposition}

\newcommand{\Z}{\mathbb{Z}}
\newcommand{\N}{\mathbb{N}}

\newcommand{\script}{\mathcal}
\newcommand{\parentheses}[1]{{\left( {#1} \right)}}

\newcommand{\p}{\parentheses}

\newcommand{\closure}[1]{\overline{#1}}

\newcommand{\Set}[1]{{\left\lbrace {#1} \right\rbrace}}
\newcommand{\singleton}{\Set}

\def\set#1:#2{\Set{{#1} \colon {#2}}}

\title{$n$-Arc Connected Graphs}
\author{Paul Gartside}
\address{Department
    of Mathematics, University of Pittsburgh, Pittsburgh, PA~15260, USA}
\email{gartside@math.pitt.edu} 

\author{Ana Mamatelashvili}
\address{School of Mathematics and Statistics, University of Melbourne, Melbourne,  VIC~3010, Australia}
\email{ana.mamatelashvili@unimelb.edu.au}

\author{Max Pitz}
\address{Department of Mathematics, University of Hamburg, Bundesstra\ss e 55, 20146 Hamburg, Germany}
\email{max.pitz@uni-hamburg.de}



\begin{document}

\begin{abstract}
Given a graph $G$, of arbitrary size and unbounded vertex degree, denote by $|G|$ the one-complex associated with $G$. The topological space $|G|$ is \emph{$n$-arc connected} ($n$-ac) if every set of no more than $n$ points of $|G|$ are contained in an arc (a homeomorphic copy of the closed unit interval). 

For any graph $G$, we show the following are equivalent: (i) $|G|$ in $7$-ac, (ii) $|G|$ is $n$-ac for all $n$, and (iii) $G$ is a subdivision of one of nine graphs. A graph $G$ has $|G|$ $6$-ac if and only if either $G$ is one of the nine $7$-ac graphs, or, after suppressing all degree-2-vertices, the graph $G$ is $3$-regular, $3$-connected, and removing any $6$ edges does not disconnect $G$ into $4$ or more components. 

Similar combinatorial characterizations of graphs $G$ such that $|G|$ is $n$-ac for $n=3, 4$ and $5$ are given. Together these results yield a complete classification of $n$-ac graphs, for all $n$.
\end{abstract}

\maketitle

%\paragraph{TO DO's}
%Align notation. Edges should be $e,f$. Vertices should be $v,w$. Points of the $1$-complex should be $x,y$. Arcs/simple closed curves should always be $\alpha,\beta$.

\section{Introduction}

Graphs are typically considered as combinatorial objects: a set of vertices, along with a set of un-ordered pairs of vertices, forming  edges abstractly connecting the vertices; but it is equally natural to consider graphs as geometric objects with a set of vertices and some pairs of vertices literally connected by an arc (a homeomorphic copy of the closed unit interval). 
Indeed right at the birth of graph theory, with Euler's solution of the K\"{o}nigsberg Bridges problem -- asking for a particular kind of physical path -- and a little later with Hamilton's solution of his Icosian problem -- requesting an arc, or circle, in the skeleton of a dodecahedron, containing all vertices -- the geometric view is the most immediate. 
When we think of a graph geometrically (a $1$-complex), then the points on edges become \emph{first class citizens}, and this change in perspective opens up new classes of problems. In this paper we always consider a given combinatorial graph as a geometric graph with the natural underlying set, topologized in any way so that each arc forming an edge has its usual topology as a subspace (the exact topology on the graph will not turn out to be important here).

A natural extension of Hamilton's problem is to ask, for some $n$, which graphs $G$ are $n$-arc Hamilton (respectively, $n$-Hamilton) that for any choice of at most $n$ vertices there is an arc (respectively, a circle) in $G$ containing the specified points. For example, a classical theorem in graph theory of Dirac \cite[Satz~9]{dirac} says that a $n$-connected graphs are $n$-Hamilton. 
(Recall that a combinatorial graph is \emph{$k$-connected} if deleting at most $k-1$ vertices does not result in a disconnection.)
However, high connectivity isn't always necessary, indeed every cycle  is $n$-Hamilton, for all $n$, despite it being only $2$-connected. Dirac also noted in \cite{dirac} that  $n$-connected graphs are not necessarily $(n+1)$-Hamilton. A characterization was  found in \cite{watkinsmesner}: let $G$ be an $n$-connected graph, $n \geq 3$, then $G$ is $(n+1)$-Hamilton if and only if no set $T$ of vertices of $G$ of size $n$ separates $G$ into more than $n$ components. As is well-known, despite the existence of simple sufficient conditions, there is still no characterization of Hamiltonicity. Let us also note that Egawa, Glas \& Locke \cite{EGL} gave a sufficient condition for an $n$-connected graph to be $(n+1)$-arc Hamilton, but the authors are not aware of characterizations of the $(n+1)$-arc Hamilton, $n$-connected graphs.

Taking the geometric viewpoint we are led to consider connecting \emph{arbitrary} points in a graph by arcs or
circles. Let $G$ be a graph, considered as a topological space, and $S$ a subset of $G$, then $(S,G)$ is \emph{$n$-arc connected} (or, $S$ is \emph{$n$-ac} in $G$) if for any choice of at most $n$ elements of $S$ there is an arc in $G$ containing the specified points, while $(S,G)$ is \emph{$n$-circle connected} (or, $S$ is \emph{$n$-cc} in $G$) if for any choice of at most $n$ elements of $S$ there is a simple closed curve in $G$ containing the specified points. 
Further, we say $S$ is $\omega$-ac (respectively, $\omega$-cc) in $X$ if it is $n$-ac ($n$-cc) in $X$ for all integers $n \in \N$.
Observe that a graph $G$ with vertices $V$ has a Hamiltonian path (respectively, cycle) if and only if $V$ is $|V|$-ac (respectively, $|V|$-cc) in $G$, and is $n$-arc Hamilton (respectively, $n$-Hamilton) if and only if $V$ is $n$-ac (respectively, $n$-cc) in $G$.

Define a graph $G$ to be $n$-ac (resp., $n$-cc) if $G$ is $n$-ac (resp., $n$-cc) in $G$ -- in other words, for any choice of at most $n$ points of $G$ there is an arc (resp., circle) in $G$ containing the specified points. 
In this paper we give a complete solution to the problem of characterizing which graphs are $n$-ac or $n$-cc. By `complete' we mean for any $n$, and for any graph, without restriction on the number of vertices, or edges, or the degree of any vertex. 
Our characterizations give tests for a graph to be $n$-ac or $n$-cc which are combinatorial in nature, only referring to vertices and edges, and which are polynomial in the number of vertices for finite graphs. 
The proofs largely rely on Menger-type results, and arguments based on (and in some cases, extending) the theory of alternating walks. 

In describing our results, we should start by stating that it is straightforward to see that a graph is $2$-cc if and only if it is $2$-connected, while the only $3$-cc graphs are cycles (see Theorem~\ref{thm_27} and preceding discussion). Hence our focus is on $n$-ac graphs, for some $n$. It is also clear that a graph $G$ is $2$-ac if and only if it is connected (combinatorially). 
In \cite{acpaper} it was shown that a non-degenerate finite connected graph $G$ is $7$-ac if and only if it is $\omega$-ac if and only if $G$ is homeomorphic to one of $6$ graphs (arc, circle, lollipop, $\theta$-curve, figure-of-eight, dumbbell).
 Extending that argument shows that this behaviour  occurs also for arbitrary graphs. Indeed (see Theorem~\ref{7ac1complex}), a non-degenerate graph $G$ is $7$-ac if and only if it is  $n$-ac for all $n$, and if and only $G$ is homeomorphic to a finite list of graphs, namely homeomorphic to one of the six finite $7$-ac graphs, or one of the finite $7$-ac graphs with some endpoints removed (giving three more possibilities).

It remains, then, to characterize the $n$-ac graphs for $n=3,4,5$ and $6$. In \cite{acpaper} infinite families of finite graphs which are $n$-ac but not $(n+1)$-ac were given for all $2 \le n \le 6$, and the problem of characterizing finite $n$-ac graphs for $n=3,4,5$ and $6$ was raised. 
Theorems~\ref{thm_27} and~\ref{thm_3acChain}, \ref{cyclic_4ac} and~\ref{chain_4-ac},  \ref{chain_5-ac} and~\ref{thm_5acChar}, and \ref{thm_6acChar} solve that problem for arbitrary graphs.  
As a sample: a graph $G$ is $6$-ac if and only if either $G$ is one of the nine $7$-ac graphs mentioned above, or, after suppressing all degree-2-vertices, the combinatorial graph $G$ is $3$-regular, $3$-connected, and removing any $6$ edges does not disconnect $G$ into $4$ or more components.

\section{Preliminaries}
\label{sec_1}

In this paper, the term \emph{graph} refers to a combinatorial graph $G=(V,E)$ where $V$ is a (possibly infinite) set and $E \subseteq [V]^2$. However, to every combinatorial graph $G=(V,E)$ we associate the topological space, $|G|$, which is the $1$-complex of $G$, namely the quotient space $\p{V \oplus \bigoplus_{e \in E} [0,1]_e } / \sim$ where $V$ carries the discrete topology and for an edge $e = \Set{v,w} \in E$ we identify $v$ in $V$ with the $0 \in [0,1]_e$ and $w$ with $1 \in [0,1]_e$. 
In fact our results hold whenever the set $|G|$ is given a topology in which the image under the quotient of each $[0,1]_e$ is homeomorphic to $[0,1]$.  (For example, the metric topology induced by vertex distance, extended to interior points of edges in the natural manner, would work equally well. The quotient topology is simply the finest one satisfying this property.) 
Where no confusion can arise -- when we discuss purely topological notions, for example -- we abuse notation and simply write $G$ for $|G|$.

\subsection{Notation and Conventions} 
Let $G=(V,E)$ be a graph. An edge, $e=\Set{v,w}$ is often abbreviated, $e=vw$. By convention we label \emph{fixed} vertices by $a, b, \ldots$, and \emph{general} vertices as $v,w$ \textit{et cetera}. 
Let $V = A \dot\cup B$ a partition of its vertex set. The set $E(A,B)$ of edges of $G$ with one endpoint in $A$ and the other in $B$ is called an \emph{edge-cut} of $G$.

For $A \subseteq V$ or $F \subseteq E$ we write $G[A]$ and $G[F]$ for the induced subgraph of $G$.

A subset $e$ of $|G|$ is called a \emph{closed edge} if it is the image under the quotient map of some $[0,1]_e$, and is called an \emph{edge} if it is a closed edge minus its endpoints. In particular note that edges are  open sets. If $e = vw\in E(G)$ then $\closure{e} = \Set{v,w} \cup e \subseteq |G|$ is the closed edge in $|G|$ naturally associated with the combinatorial edge $e$ in $G$. By convention we label points in the space $|G|$ by $x,y, \ldots$.

\subsection{Background from graph theory}

Our main tools from graph theory will be the block-cutvertex decomposition of (possibly infinite) connected graphs, and certain variants of Menger's theorem, especially the ones involving the concept of alternating walks. A good overview of these techniques is given, for example, in the chapter on connectivity of Diestel's book, \cite[\S~3.1 \& 3.3]{Diestel}. 

%A graph is \emph{degenerate} if it has one vertex (and no edges). 
 A graph $G=(V,E)$ is \emph{cyclically connected} if every two vertices lie on a cycle, and, recall, is \emph{$2$-connected} if removing any single vertex does not disconnect the graph. Note that a graph is cyclically connected if,  in our terminology, the vertices are $2$-cc in the graph.
  Observe that, according to this definition, the complete graph on two vertices is $2$-connected but not cyclically connected. However, by Menger's theorem stated below, any $2$-connected graph with at least $3$ vertices is cyclically connected.

\subsubsection{Block-cutvertex decomposition} Let $G$ be a  connected graph. A \emph{block} of $G$ is an inclusion-maximal $2$-connected subgraph. Every edge is contained in a block, and by their maximality, different blocks overlap in at most one vertex, which then must be a cut-vertex of $G$. Therefore, the blocks form an edge-disjoint decomposition of $G$. Let $C\subseteq V$ denote the set of cut-vertices of $G=(V,E)$ and $\mathcal{B}$ the collection of blocks. The \emph{block graph} $B(G)$ of $G$ is the bipartite graph formed on the vertex set $C \dot\cup \script{B}$ with an edge $cB \in E(B(G))$ if and only if $c \in B$. For the proof of the next lemma see \cite[Lemmas~3.1.3 \& 3.1.4]{Diestel}.

\begin{lemma}
\label{lem_blockgraph}
The block graph of a (possibly infinite) connected graph is a (possibly infinite) graph-theoretic tree. 
\end{lemma}

\subsubsection{Chain graphs and cycle graphs} A connected graph is called a \emph{chain graph} if it is a linearly ordered union (possibly just one) of subgraphs (called \emph{links}) such that only consecutive graphs meet, and their intersection consists of a single vertex only (called the \emph{linking vertex}). Thus, a connected graph $G$ is a chain graph of $2$-connected links if and only if its block graph $B(G)$ is a finite path, a ray, or a double ray. 

Similarly, a connected graph $G$ is a \emph{cycle graph} if $G$ is a union of graphs $L_0, L_1, \ldots,   L_{n-1}$ (called \emph{links}) for some $n\in \N$ where (a) $L_i \cap L_j = \emptyset$ if $|i-j| > 1 \mod n$ and (b)  $L_i$ meets $L_j$ at a single vertex (called the \emph{linking vertex}) if $|i-j|=1 \mod n$.  

\subsubsection{Menger's theorem and alternating walks}

We say a graph $G=(V,E)$ is $k$-connected for some $k\in \N$ if the induced subgraph $G-W$ is connected for all $W \subseteq V$ with $|W|<k$. We note that even if a graph $G$ is $k$-connected for some large $k$, the underlying topological 1-complex $|G|$ is never $3$-connected in the topological sense, as removing two endpoints of an edge disconnects the interior of that edge from the rest of the graph.

Given sets $A,B$ of vertices, we call a path [walk] $P=x_0, \ldots, x_n$ an $A-B$ \emph{path} [\emph{walk}] if $V(P) \cap A = \Set{x_0}$ and $V(P) \cap B = \Set{x_n}$, i.e.\ if the path [walk] starts in $A$, ends in $B$ and is otherwise disjoint from $A \cup B$. Two or more paths are \emph{independent} if none of them contains an inner vertex of another. If $A,B \subseteq V$ and $X \subseteq V$ is such that every $A-B$ path in $G$ contains a vertex from $X$, then we say that $X$ \emph{separates} the vertex set $A$ from $B$. This implies $A \cap B \subseteq X$, i.e.\ $X$ does not have to be disjoint from $A$ or $B$.

\begin{theorem}[Menger's Theorem]
\label{menger}
Let $G=(V,E)$ be a (potentially infinite) graph and $A,B \subseteq V$. Then the minimum number of vertices separating $A$ from $B$ in $G$ is equal to the maximum number of disjoint $A-B$ paths in $G$.
\end{theorem}

\begin{proof}
See \cite[Theorem~3.3.1 \& Prop~8.4.1]{Diestel} respectively for the proof for finite graphs and its infinite extension.
\end{proof}

%XXX state corresponding result for independent family of paths XXX

However, at several points in this paper we shall employ the following more algorithmic version of Menger's theorem using the notion of alternating walks. Recall that for given sets $A,B$ of vertices and $\script{P}$ a collection of disjoint $A-B$ paths, a walk $W=x_0 e_0x_1 e_1 \ldots e_{n-1}x_n$ in $G$ with no repeated edges (but possibly repeated vertices) is \emph{alternating} with respect to $\script{P}$ if 
\begin{itemize}
\item $W$ starts in $A \setminus \bigcup \script{P}$,
\item the only repeated vertices of $W$ lie on paths in $\script{P}$,
\item whenever $W$ hits a vertex of some path $Q \in \script{P}$ (meaning say $e_i \notin Q$ but $x_{i+1} \in Q$) then $P$ follows $Q$ back towards the direction of $A$ for at least one edge of $Q$, and
\item whenever $W$ uses an edge $e$ from a path $Q \in \script{P}$, then $W$ traverses this edge backwards.
\end{itemize} 
We refer the reader to \cite[Fig.~3.3.2]{Diestel} and the surrounding discussion for further information. The following two lemmas list the two crucial properties of alternating paths in the context of Menger's theorem. For the proof of the first see \cite[Lemma~3.3.3]{Diestel}.

\begin{lemma}
\label{lem_existalt1}
If no alternating walk ends in $B \setminus \bigcup \script{P}$, then $G$ contains an $A-B$ separator $X$ on $\script{P}$ with $|X| = |\script{P}|$.
\end{lemma}

\begin{lemma}
\label{lem_existencealternatingwalk}
If an alternating walk $W$ ends in $B \setminus \bigcup \script{P}$, then $G$ contains a set of disjoint paths $\script{P}'$ with $|\script{P}'|= |\script{P}| + 1$. Moreover, the alternating walk $W$ can be chosen such that 
\begin{enumerate}
\item $E(\script{P}')$ is precisely the symmetric difference $E(\script{P}) \triangle E(W)$, and
\item every path $P' \in \script{P}'$ traverses the edges of $E(P') \cap E(W)$ in the same order as $W$. 
\end{enumerate}
\end{lemma}
\begin{proof}
The first assertion of the lemma is proved in \cite[Lemma~3.3.2]{Diestel}. However, in order to see the moreover-part, we need to recall the main idea of \cite[Lemma~3.3.2]{Diestel}. First, the definition of an alternating walk ensures that after taking the symmetric difference  $\Delta=E(\script{P}) \triangle E(W) := \bigcup_{P \in \script{P}} E(P) \triangle E(W)$, 
every vertex outside $A \cup B$ will have degree $0$ or $2$ in $G[\Delta]$, and vertices in $A$ or $B$ lying on $\script{P}$ of $W$ will continue to have degree $1$. Thus, every component of $G[\Delta]$ containing a vertex $a \in A$ will have to be a finite path starting at $a$ and ending at $A \cup B$. To see that any such path in fact has to end at a vertex of $B$, one proves the additional fact that the path traverses an edge in the symmetric difference always in the forward direction with respect to $\script{P}$ or $W$. This, clearly, yields the first assertion of the lemma.

However, there might be further components of $G[\Delta]$ which are finite cycles not incident with $A \cup B$. To eliminate the occurrence of such finite cycles (and hence to establish property (1)), we choose an alternating walk $W$ ending in $B \setminus \bigcup \script{P}$ such that $|E(W) \setminus \bigcup_{P \in \script{P}} E(P)|$ is minimal. Now suppose for a contradiction that there is a cycle $C$ in $G[\Delta]$. By the argument above, we know that traversing the edges of $C$ in the forward direction with respect to $\script{P}$ or $W$ induces a cyclic order $<$ on $E(C)$. Observe that since $(C,<)$ is a cyclic order, there are edges $e_k,e_{k+i}$ for $i \geq 0$ of $E(W)$ such that 
$e_{k+i} < f_1 < f_1 < \cdots < f_\ell < e_k$ 
is a segment of $(E(C),<)$, where $\ell \geq 0$ and $f_0,\ldots,f_\ell$ is a subpath of some path $P \in \script{P}$ (here, $\ell=0$ allows for the possibility that $e_{k+i} < e_k$ are successors in $(E(C),<)$).

If $\ell > 0$, then the walk $W' = x_0 e_0 \ldots e_{k-1} x_k f_\ell \ldots f_1 x_{k+i+1} e_{k+i+1} \ldots e_{n-1} x_n$ is an alternating walk contradicting the minimality of $W$. Otherwise, if $\ell = 0$ (so $i>0$), write $x$ for the vertex incident with both $e_{k+i}$ and $e_k$. Then by definition of alternating, there is a path $Q \in \script{P}$ with $Q=y_0 f_0y_1 f_1 \ldots f_{m-1}y_m$ such that $e_{k+i+1}=f_r$, $e_{k-1} = f_{r+1}$ and $x = y_{r+1}$. But then the walk $W' = x_0 e_0 \ldots e_{k-1} x e_{k+i+1} \ldots e_{n-1} x_n$ is an alternating walk contradicting the minimality of $W$. This contradiction shows that $\Delta$ is cycle-free, establishing (1).

It remains to argue that by choosing $|E(W) \setminus \bigcup_{P \in \script{P}} E(P)|$ minimal, we also have property (2). But if (2) fails for some path $P' \in \script{P}'$, there a segment on the path $P'$ of the form $e_{k+i} < f_1 < f_1 < \cdots < f_\ell < e_k$ where $\ell \geq 0$ and $f_0,\ldots,f_\ell$ is a subpath of some path $P \in \script{P}$, which yields a contradiction to the minimality of $W$ as before. 
\end{proof}

From Lemmas~\ref{lem_existalt1} and \ref{lem_existencealternatingwalk} we immediately deduce:
\begin{theorem}
\label{thm_existencealternatingwalk}
Let $G=(V,E)$ be a (potentially infinite) $k$-connected graph, $A,B \subseteq V$ be disjoint sets of vertices each of size at least $k$, and $\script{P}$ a collection of disjoint $A-B$ paths with $|\script{P}|=i<k$. Then there exists an alternating $A-B$ walk $W$ such that the symmetric difference $E(\script{P}) \triangle E(W)$ is precisely the edge set of a collection of $i+1$ many disjoint $A-B$ paths.
\end{theorem}

\subsection{Background on \texorpdfstring{$n$}{n}-arc connectedness}

%The following lemma has appeared in the previous paper \cite{acpaper}.

\begin{lemma}[{\cite[Lemma 2.6 ]{acpaper}}]
\label{tripodlemma}
Let $v$ be a vertex of a graph $G$ of degree at least $3$. If $x_0,x_1,x_2$ are interior points of distinct edges of $G$ incident with $v$, then any arc $\alpha$ containing $\{ x_0, x_1, x_2 \}$ satisfies $v \in int(\alpha)$, and one of its endpoints lies in $[v,x_0]\cup [v,x_1]\cup [v,x_2]$. \qed
\end{lemma}

The next lemma can be seen as a partial extension of the previous lemma:

\begin{lemma}
\label{lem_oddevencuts}
Let $G$ be a graph, and $E(A,B) = \{e_1=a_1b_1, \ldots, e_n=a_nb_n \}$ be an edge cut of $G$ with $a_i \in A$ and $b_i \in B$. Suppose that $x_i \in e_i$ are interior points. %Write $G_A = G[A] \cup \bigcup [a_i,x_i]$ and $G_B = G[B] \cup \bigcup [x_i,b_i]$.

Then any arc containing $x_1, \ldots, x_n$ with endpoints $x_1$ and $x_n$ contains either both $[a_1,x_1]$ and $[a_n,x_n]$ or both $[x_1,b_1]$ and $[x_n,b_n]$ if $n$ is even, and it contains $[a_1,x_1]$ and $[x_n,b_n]$ or vice versa if $n$ is odd. 
\end{lemma}

\begin{proof}
From the given arc, fix an embedding $\alpha : [0,1] \to G$ such that $\alpha (0)=x_1$, $\alpha (1)=x_n$ and $x_2, \ldots , x_{n-1}$ are in $\alpha ([0,1])$. For concreteness let us suppose that $n$ is even and the arc $\alpha$ travels from $x_1$ along the edge $e_1$ to $a_1$ (rather than $b_1$), and so the given arc  contains $[a_1,x_1]$. We show the arc also contains $[a_n,x_n]$. The other cases are similar. 

After $a_1$, which is in $A$, the arc $\alpha$ passes through the even number of points $x_2, \ldots, x_{n-1}$ in some order, before ending at $x_n$. As it does so the arc must cross backwards and forwards between $A$ and $B$ an even number of times. Hence $\alpha$ must enter $e_n$ from $A$, in other words by passing through $a_n$, and thus the arc contains $[a_n,x_n]$, as claimed.
\end{proof}

Note that when picking an arc witnessing  that points $x_0, x_1, \ldots, x_{n-1} \in X$ lie on a common arc, we may assume that endpoints of the arc are among $x_0, x_1, \ldots, x_{n-1}$.

\begin{lemma}\label{cutting_set}
Let $X$ be a topological space. If there exists $A \subseteq X$ such that $|A| \leq n \in \N$ and $X\backslash A$ has at least $n+2$ components, then $X$ is not $(n+2)$-ac.  
\end{lemma}
\begin{proof}
Pick $n+2$ points $x_0, x_1, \ldots, x_{n+1}$ each belonging to a distinct component of $X\backslash A$. Suppose there is an arc  containing $x_0, x_1, \ldots, x_{n+1}$. Relabeling $x_0, x_1, \ldots, x_{n+1}$ if necessary, we can fix an embedding $\alpha : I \to X$ and  $0 = t_0 < t_1 < \cdots < t_{n+1} = 1$ such that $x_i = \alpha(t_i)$ for each $i$. Then $\alpha((t_i, t_{i+1})) \cap A \neq \emptyset$ for each $i = 0, 1, \ldots, n$, which is a contradiction since $|A| \leq n$ and $\alpha$ is injective.  
\end{proof}

\begin{lemma}
\label{lem_cutting_set}
Let $G$ be a graph. If for some $n \in \N$, the condition
\[(\star_n) \quad \quad \text{no $n$ points of $|G|$ cut $|G|$ into at least $n+2$ components} \]
holds, then $(\star_n)$ implies $(\star_m)$ for all $1 \leq m \leq n$.
\end{lemma}

\begin{proof}
Let $A \subseteq |G|$ be finite of size $m \geq 1$ witnessing the failure of $(\star_m)$. Then  $|G| \setminus A$ contains at least one half-open edge, i.e.\ an open set $U$ such that $U \cong (0,1)$ and $\closure{U} \cong [0,1)$. By picking $n-m$ many points from $U$ and adding them to the set $A$, we obtain a set $A'$ witnessing the failure of $(\star_n)$.
\end{proof}

Our last lemma in this section says that when verifying whether a graph $G$ is $n$-ac, it suffices to consider points on the interior of edges of $|G|$.

\begin{lemma}
\label{lem_wlogpointsonedges}
For $n \in \N$, a graph $G$ is $n$-ac if and only if $(|G| \setminus V,|G|)$ is $n$-ac.
\end{lemma}
\begin{proof}
Only the backwards implication requires proof. Assume that $(|G| \setminus V,|G|)$ is $n$-ac and let $x_0, x_1, \ldots, x_n \in G$ be arbitrary. Pick $y_0, y_1, \ldots, y_n \in |G| \setminus V$ as follows: if $x_i$ lies on the interior of an edge, then $y_i = x_i$; otherwise, if $x_i \in V$, let $y_i$ be a point on the interior of some edge incident with $x_i$. By assumption, there is an arc $\alpha$ containing $y_0, y_1, \ldots, y_n$ which we may assume to have endpoints $y_0$ and $y_n$. Therefore, $x_1, x_2, \ldots, x_{n-1} \in \alpha$. We will show that $\alpha$ can be extended to include $x_0$ and the same argument will work for $x_n$ as well. If $y_0 = x_0$ we are done. If $y_0 \neq x_0$ then $x_0$ is an endpoint of the edge containing $y_0$, say $e$. The arc $\alpha$ contains one of these endpoint and if $x_0 \in \alpha$ we are done. Otherwise, $\alpha \cup e \cup \Set{x_0}$ is also an arc and contains $x_0$.   
\end{proof}

\section{Characterizing \texorpdfstring{$n$}{n}-ac Graphs} 

\subsection{Characterizing \texorpdfstring{$2$}{2}-ac, \texorpdfstring{$2$}{2}-cc, \texorpdfstring{$3$}{3}-ac and \texorpdfstring{$3$}{3}-cc graphs} 
A graph $G$ is \emph{$n$-strongly arc connected}, abbreviated $n$-sac (see \cite{sacpaper}) if for any list of no more than $n$ elements of $|G|$ there is an arc in $G$ containing the points in the specified order. We note that no graph is $4$-sac (pick four points $x_1, x_3, x_2, x_4$ in that order along any edge). It is evident that the following are  equivalent for a graph $G$: (i) $G$ is $2$-ac, (ii) $G$ is $2$-sac, (iii) $|G|$ is connected, and (iv) $G$ is connected (combinatorially). 

It is also clear that a graph is $3$-cc if and only if it is a cycle. Indeed a  graph $G$ is not $3$-cc if (i) it contains a vertex of degree one (that vertex is not in any circle), or (ii) a vertex of degree at least $3$ (consider three points from the interior of three edges exiting the vertex), or (iii) is a chain. 

We characterize $3$-sac and $2$-cc graphs.  
The equivalence of (1) through (4) below for \emph{finite} graphs  was established in \cite[Prop.~6]{sacpaper}. 
\begin{theorem}
\label{thm_27}
For a (possibly infinite) graph $G$, the following are equivalent:
\begin{enumerate}
\item $G$ is $3$-sac,
\item $G$ is cyclically connected,
\item any three points of $|G|$ lie on a circle or a $\theta$-curve, 
\item $G \neq K_2$ is $2$-connected, and
\item $G$ is $2$-cc.
\end{enumerate} 
\end{theorem}

\begin{proof}
The equivalence of (2) $\Leftrightarrow$ (4) follows from Menger's Theorem~\ref{menger}.

For (2) $\Rightarrow$ (3), pick three points $x_0,x_1,x_2 \in G$. Now for every each 2-element subset $A_i$ of $\{x_0,x_1,x_2\}$, use the fact that $G$ is cyclically connected to find a (finite) cycle $C_i \subseteq G$ containing the two points of $A_i$. Then consider the finite connected subgraph $H = \bigcup_i C_i$ of $G$. By construction, any two points of $H$ lie on a cycle, so $H$ is cyclically connected. By the finite case, the three points $x_0,x_1,x_2$ of $H$ lie on a circle or a $\theta$-curve in $H$, and hence in also $G$.

The implication (3) $\Rightarrow$ (1) follows from the finite case (see \cite[Prop.~6]{sacpaper}), and to see (1) $\Rightarrow$ (4), note that if a topological space has a cut-point, then it fails to be $3$-sac, \cite[Lemma~1]{sacpaper}.

Finally, evidently $2$-cc graphs are cyclically connected, while (3) $\Rightarrow$ (5) since the circle and $\theta$-curve are clearly $2$-cc. 
\end{proof}

Thus cyclically connected graphs are (strongly) $3$-ac, and this  extends naturally to a characterization of $3$-ac graphs.

\begin{theorem}
\label{thm_3acChain}
A (potentially infinite) graph $G$ is $3$-ac if and only if it is a chain graph of $2$-connected links, or, equivalently, if and only if its block graph is connected and contains no vertex of degree at least $3$.
\end{theorem}

\begin{proof}
To see that the conditions are necessary, consider the block-cutvertex decomposition of $G$ and its associated block graph $B(G)$, which is a (potentially infinite) tree by Lemma~\ref{lem_blockgraph}. To prove that $G$ is a chain graph of $2$-connected links, it suffices to show every vertex of $B(G)$ has degree at most $2$. It follows from Lemma~\ref{cutting_set} that no cut-vertex of $G$ can have degree strictly bigger than $2$ in $T$. And if there is a block $B$ of $G$ with contains at least three cut vertices $c_0,c_1,c_2$, then picking three points $x_i$ each on the interior of edges $e_i \in G \setminus E(B)$ incident with $c_i$ easily shows that $G$ cannot be $3$-ac. 

For the converse direction, suppose $G$ is a chain graph of $2$-connected graphs. Pick $x_0,x_1,x_2$ in $G$. Then there is a minimal finite `convex' part of that chain, say $L_0, \ldots, L_n$ with $L_i \cap L_{i+1} = v_i$,  covering $x_0,x_1,x_2$.

If all $x_0,x_1,x_2$ lie in the same link $L_0$, we are done by Theorem~\ref{thm_27}, as $L_0$ is $2$-connected. If $x_0,x_1$ lie in same link, say $L_0$, and $x_2$ lies in $L_n$, then we may find 
\begin{itemize}
\item an arc $\alpha_0$ in $L_0$ that picks up $\Set{x_0,x_1,v_0}$ ending at $v_0$ (clear if $L_0 = K_2$, and by Theorem~\ref{thm_27} otherwise),
\item arcs $\alpha_i$ in $L_i$ with endpoints $v_{i-1}$ and $v_i$ for $ 0 < i < n$, and 
\item an arc $\alpha_n$ in $L_n$ with endpoints $v_{n-1}$ and $x_3$.
\end{itemize} 
It is then clear that the concatenation of the $\alpha_i$ is an arc though our three points $x_0,x_1,x_2$. Finally, in the case where $x_0 \in L_0$, $x_1 \in L_i$ for $0<i<n$ and $x_2 \in L_n$, the same approach extends straightforwardly.
\end{proof}

\subsection{Characterizing \texorpdfstring{$4$}{4}-ac graphs}

Let us say that a graph $G$ is a \emph{basic $4$-ac graph} if $G$ is (a subdivision of) a circle, a $\theta$-curve, a cycle graph of two circles and an arc (`happy-face curve'), or if it is (a subdivision of) a cycle graph of alternating two circles and two arcs (`baguette curve'). See the following sketch for the latter two basic $4$-ac graphs. 

\vspace{6pt}
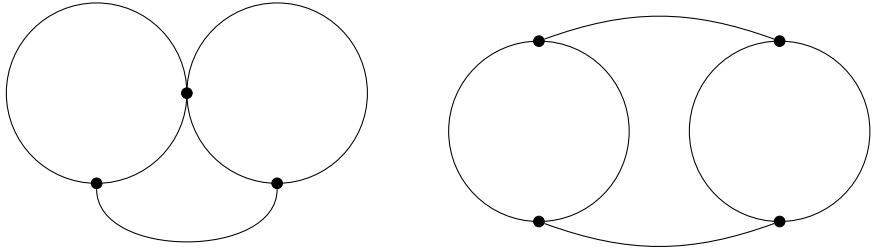
\begin{minipage}{.49\textwidth}
\begin{center}
\begin{tikzpicture}[scale=.8, use Hobby shortcut]
\begin{scope}[xshift=-1.5 cm]
\draw (0,0) circle (1.5cm);

\node[circle, draw, inner sep=0.6pt, minimum width=4pt, fill=black]  (b1) at (-90:1.5cm) {};
\end{scope}

\begin{scope}[xshift=1.5 cm]
\node[circle, draw, inner sep=0.6pt, minimum width=4pt, fill=black]  (a) at (-1.5,0) {};

\draw (0,0) circle (1.5cm);

\node[circle, draw, inner sep=0.6pt, minimum width=4pt, fill=black]  (b2) at (-90:1.5cm) {};
\end{scope}

\draw (b1) to[out=-90,in=-90] (b2);

%\coordinate [label=left:$e_1$] (e1) at (-1.35,0);
%\coordinate [label=left:$e_2$] (e2) at (-0.6,-.2);
%\coordinate [label=right:$e_3$] (e3) at (0.6,-.2);
%\coordinate [label=right:$e_4$] (e4) at (1.35,0);
%\coordinate [label=below:$e_5$] (e5) at (0.9,-1);

\end{tikzpicture}
\end{center}
\end{minipage}
\begin{minipage}{.49\textwidth}
\begin{center}
\begin{tikzpicture}[scale=.8, use Hobby shortcut]
\begin{scope}[xshift=-1.5 cm]
\draw (0,0) circle (1.5cm);
\node[circle, draw, inner sep=0.6pt, minimum width=4pt, fill=black]  (a1) at (90:1.5cm) {};
\node[circle, draw, inner sep=0.6pt, minimum width=4pt, fill=black]  (b1) at (-90:1.5cm) {};
\end{scope}

\begin{scope}[xshift=2.5 cm]

\draw (0,0) circle (1.5cm);
\node[circle, draw, inner sep=0.6pt, minimum width=4pt, fill=black]  (a2) at (90:1.5cm) {};
\node[circle, draw, inner sep=0.6pt, minimum width=4pt, fill=black]  (b2) at (-90:1.5cm) {};
\end{scope}

\draw (a1) to[out=20,in=160] (a2);
\draw (b1) to[out=-20,in=-160] (b2);

\end{tikzpicture}
\end{center}
\end{minipage}
\vspace{6pt}

As any four edges of these graphs either lie on a common $\theta$-curve, a figure-8-curve or a dumbbell, these graphs are indeed $4$-ac. The purpose of our next theorem is to prove a `converse' of this observation for cyclically connected 4-ac graphs. 

\begin{theorem}\label{cyclic_4ac}
For a cyclically connected graph  $G$, the following are equivalent: 
\begin{enumerate}
\item $G$ is $4$-ac, 
\item no two vertices cut $G$ into $4$ or more components, and  
\item any four edges of $G$ are contained in a basic $4$-ac subgraph of $G$. 
\end{enumerate}
\end{theorem}

\begin{proof}
The implication $(1) \Rightarrow (2)$ is Lemma~\ref{cutting_set}, and $(3) \Rightarrow (1)$ is clear.

For $(2) \Rightarrow (3)$, let $G$ be a $4$-ac cyclically connected graph such that no two point set cuts it into at least $4$ components. Let $x_0, \ldots, x_3 \in G$. To show that $G$ is $4$-ac, we may assume, by Lemma~\ref{lem_wlogpointsonedges}, that all $x_i$ are interior points of edges. 

By Theorem~\ref{thm_27}, the three points $x_0,x_1,x_2$ lie on a common circle or a common $\theta$-curve $X$. In the first case, Menger's Theorem~\ref{menger}---applied with the two endvertices of the edge containing $x_3$ against $V(X)$---shows that there are two disjoint $x_3-X$ arcs, and so there is a $\theta$-curve contain $\Set{x_0,\ldots,x_3}$ and we are done. 

Otherwise, let us write $a$ and $b$ for the two degree-3-vertices of the $\theta$-curve $X$, and $e_0,e_1,e_2$ for its three edges. Further, as $x_0,x_1,x_2$ do not lie on a common cycle, we may label the edges of $X$ such that $x_i \in e_i$. Since $G$ is cyclically connected, it follows from Menger's Theorem~\ref{menger} as before that there is an arc $\alpha$ such that $x_3 \in \alpha$ and $X\cap \alpha = \{ \alpha(0), \alpha(1) \}$. Up to symmetry, the following cases can occur:
\begin{center}
\begin{tabular}{ll}

(1) \ $\alpha(0),\alpha(1) \in e_0$, & 
(2) \ $\alpha(0) \in e_0$, $\alpha(1) \in e_1$, \\
(3) \ $\alpha(0) = a$, $\alpha(1) \in e_0$, or \ \ & 
(4) \ $\alpha(0)= a$, $\alpha(1) = b$.
\end{tabular} 
\end{center}

In the first case, $Y=X\cup \alpha$ is homeomorphic to a baguette curve. In the second case, $Y$ is homeomorphic to a $K_4$, where removing any edge not containing a point $x_i$ reduces it to a $\theta$-curve. In the third case, $Y$ is a happy-face-curve. Finally, in the last case, $Y$ consists of the vertices $a$ and $b$ with four parallel edges $e_0,\ldots,e_3$ between them. Since by assumption, $|G| \setminus \Set{a,b}$ consists of at most three components, there is an arc $\delta$ in $G \setminus \Set{a,b}$ internally disjoint from $Y$ with say $\delta(0) \in e_0$ and $\delta(1) \in e_1$. One checks that any four points in $Z = Y \cup \delta$ lie on either a $\theta$-curve or on a happy-face-curve, which completes the proof. 
\end{proof}

Next, we extend our characterization of $4$-ac graphs to graphs which are no longer necessarily cyclically connected. For this, the following lemma gives us additional control over arcs in our four basic $4$-ac graphs.

\begin{lemma}\label{basic_4-ac}
Let $G$ be one of our basic $4$-ac graphs. If $w$ is a point in the interior of an edge of $G$, then for any three further points in $G$, there exists an arc in $G$ that contains those three points and has $w$ as an endpoint. 
\end{lemma}
\begin{proof}
If $G$ is either a circle or a $\theta$ curve, then it is easy to see that the assertion of the lemma holds. 

So let $G$ be the happy-face curve with cycles $C_1,C_2$, degree-4 vertex $a \in C_1 \cap C_2$, degree-3-vertices $b \in C_1 \setminus C_2$ and $c \in C_2 \setminus C_1$ and edges $\Set{e_0,e_1}= E(C_1)$, $\Set{e_2,e_3}= E( C_2)$ and $e_4=bc$. Pick points $w, x_0, x_1, x_2 \in G$. Since removing one of $e_0, \ldots, e_3$ reduces $G$ to a $\theta$-curve, we only have to consider the case when one of $w, x_0, x_1, x_2$ belongs to the interior of each $e_i$ for $0 \leq i \leq 3$. But now, since $G \setminus e_4$ is a figure-8-curve, the assertion of the lemma is clear.

Finally, assume $G$ is the baguette curve with cycles $C_1,C_2$, degree-3 vertices $a,b \in C_1$ and $c,d \in C_2$ and edges $\Set{e_0,e_1} =E( C_1)$, $\Set{e_2,e_3} = E( C_2)$ and $e_4=ac,e_5=bd$ between $C_1$ and $C_2$. Pick points $w, x_0, x_1, x_2 \in G$. Since removing one of $e_0, \ldots, e_3$ reduces $G$ to a $\theta$-curve, we only have to consider the case when one of $w, x_0, x_1, x_2$ belongs to the interior of each $e_i$ for $0 \leq i \leq 3$. But now, since $G \setminus e_5$ is a dumbbell with $w$ lying on one of its cycles, the assertion of the lemma is again clear.
\end{proof}

\begin{theorem} 
\label{chain_4-ac}
A  graph $G$ is $4$-ac if and only if it is a chain graph such that
\begin{enumerate}
\item all links are $2$-connected and $4$-ac, 
\item all interior links are edges, 
\item\label{chain_4-ac(3)} if $v$ is a cut vertex and $L$ a link of $G$ with $v \in L$, then $\operatorname{deg}_L(v) \leq 2$.
\end{enumerate} 
\end{theorem}
\begin{proof} Suppose $G$ is a $4$-ac graph. Then $G$ is $3$-ac, and so a chain graph of $2$-connected links by Theorem~\ref{thm_3acChain}. Item (1) is now clear. 

For (2), suppose for a contradiction, there is a chain graph $G$ of $2$-connected links with decomposition $\set{L_n}:{n \in J}$ for $J \subseteq \Z$ with $|J| \geq 3$ that is $4$-ac but one of the interior links, say $L_0$, is not an arc. Consider the subgraph $G' = L_{-1} \cup L_0 \cup L_1$ where $L_{-1} \cap L_0=\{u\}$ and $L_0 \cap L_1=\{v\}$. Pick $x_0$ in $L_{-1} \setminus \{u\}$ and $x_3$ in $L_1 \setminus \{v\}$. Observe that any arc containing $x_0$ and $x_3$ and any two further points in $L_0$ must (without loss of generality) start at $x_0$  and end at $x_3$. So it suffices to show that we can choose $x_1, x_2$ in $L_0 \setminus \Set{u,v}$ so that there is no arc in $L_0$ starting at $u$, ending at $v$ and containing both $x_1$ and $x_2$. Consider $u$ in $L_0$. By $2$-connectedness of $L_0$ and the fact that $L_0$ is not an arc, we must have $\operatorname{deg}_{L_0}(u) \geq 2$. Pick $x_1$ and $x_2$ from the interior of distinct edges of $L_0$ incident with $u$. Now it is clear that no arc in $L_0$ starting at $u$ and containing $x_1$ and $x_2$, can end at $v$, a contradiction.

For (3), suppose there are links $L_0$ and $L_1$ with $L_0 \cap L_1 = \singleton{v}$ and $\operatorname{deg}_{L_0}(v) \geq 3$. Then picking three vertices on the interior of different edges incident with $v$ in $L_0$, and picking a fourth vertex on the interior of an edge incident with $v$ in $L_1$ shows that $G$ is not $4$-ac, a contradiction. 

For the converse, assume that $G$ is a chain graph satisfying properties (1)--(3). We may suppose that $G$ contains a non-trivial link $L$ not isomorphic to $K_2$. If the block graph of $G$ is infinite, it follows from $(2)$ that $G$ is isomorphic to $L$ with a one-way infinite ray $R$ attached to a vertex $v$ of $L$. But then it is clear that it suffices to show that $L$ with a single extra edge attached at $v$ is $4$-ac. Thus we may assume, by (1) and the foregoing discussion, that $G$ consists of finite number $\geq 2$ of links. So let $L_1,\ldots,L_n$ with $n\geq 2$ and $L_i \cap L_{i+1} = \Set{v_i}$ be the decomposition of $G$ into links according to properties (1)--(3). Pick four points $x_0, \ldots, x_3 \in G$. If all four points are contained in the same link, then we are done by (1). Otherwise, find basic $4$-ac subgraphs $H$ and $H'$ in $L_1$ and $L_n$ containing $v_1 \cup \p{\Set{x_0,\ldots,x_3} \cap L_1}$ and $v_{n-1} \cup \p{\Set{x_0,\ldots,x_3} \cap L_n}$ respectively. Since by (3), $v_1$ and $v_{n-1}$ have degree $2$ in $H$ and $H'$ respectively, it follows from Lemma~\ref{basic_4-ac} that there are arcs $\alpha$ and $\alpha'$ in $H$ and $H'$ picking up all vertices $\Set{x_0,\ldots,x_3} \cap L_1$ and $\Set{x_0,\ldots,x_3} \cap L_n$ and starting at $v_1$ and $v_{n-1}$ respectively. Since all middle links are arcs by (2), the arc $\alpha \cup L_2 \cup \ldots \cup L_{n-1} \cup \alpha'$ witnesses that $G$ is $4$-ac. 
\end{proof}

\subsection{Characterizing \texorpdfstring{$5$}{5}-ac graphs}

Our first theorem reduces the problem of characterizing all $5$-ac graphs to the cyclically connected case.

\begin{theorem}\label{chain_5-ac}
Let $G$ be a  graph which is \emph{not} cyclically connected. Then $G$ is  $5$-ac if and only if $G$ is homeomorphic to one of the following graphs:

(a) a finite path (equivalently, an arc), a ray or a double ray; 
(b) a lollipop with or without the endpoint;
(c) the dumbbell graph, or
(d) the figure-of-eight-graph.
\end{theorem}

See the following diagram for sketches of the lollipop graph, the dumbbell graph, and the figure-of-eight graph.

\vspace{6pt}
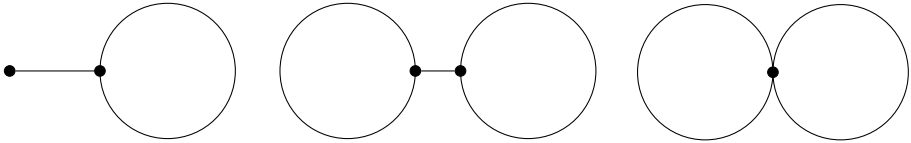
\begin{minipage}{.25\textwidth}
\begin{center}
\begin{tikzpicture}[scale=.6, use Hobby shortcut]

\draw (0,0) circle (1.5cm);

%\node[circle, draw, inner sep=0.6pt, minimum width=4pt, fill=black]  (b1) at (-90:1.5cm) {};
\node[circle, draw, inner sep=0.6pt, minimum width=4pt, fill=black]  (a) at (-1.5,0) {};

\node[circle, draw, inner sep=0.6pt, minimum width=4pt, fill=black]  (b) at (-3.5,0) {};

\draw (a) to (b);

\end{tikzpicture}
\end{center}
\end{minipage}
\begin{minipage}{.4\textwidth}
\begin{center}
\begin{tikzpicture}[scale=.6, use Hobby shortcut]
\begin{scope}[xshift=-1.5 cm]
\draw (0,0) circle (1.5cm);
\node[circle, draw, inner sep=0.6pt, minimum width=4pt, fill=black]  (a1) at (0:1.5cm) {};
%\node[circle, draw, inner sep=0.6pt, minimum width=4pt, fill=black]  (b1) at (-90:1.5cm) {};
\end{scope}

\begin{scope}[xshift=2.5 cm]

\draw (0,0) circle (1.5cm);
\node[circle, draw, inner sep=0.6pt, minimum width=4pt, fill=black]  (a2) at (180:1.5cm) {};
%\node[circle, draw, inner sep=0.6pt, minimum width=4pt, fill=black]  (b2) at (-90:1.5cm) {};
\end{scope}

\draw (a1) to (a2);
%\draw (b1) to[out=-20,in=-160] (b2);

\end{tikzpicture}
\end{center}
\end{minipage}
\begin{minipage}{.28\textwidth}
\begin{center}
\begin{tikzpicture}[scale=.6, use Hobby shortcut]
\begin{scope}[xshift=-1 cm]
\draw (0,0) circle (1.5cm);
\node[circle, draw, inner sep=0.6pt, minimum width=4pt, fill=black]  (a1) at (0:1.5cm) {};
%\node[circle, draw, inner sep=0.6pt, minimum width=4pt, fill=black]  (b1) at (-90:1.5cm) {};
\end{scope}

\begin{scope}[xshift=2 cm]

\draw (0,0) circle (1.5cm);
%\node[circle, draw, inner sep=0.6pt, minimum width=4pt, fill=black]  (a2) at (90:1.5cm) {};
%\node[circle, draw, inner sep=0.6pt, minimum width=4pt, fill=black]  (b2) at (-90:1.5cm) {};
\end{scope}

%\draw (a1) to[out=20,in=160] (a2);
%\draw (b1) to[out=-20,in=-160] (b2);

\end{tikzpicture}
\end{center}
\end{minipage}
\vspace{6pt}

\begin{proof} It is straightforward to check that each of the listed graphs is indeed $5$-ac. 
So suppose $G$ is a  non-cyclically connected but $5$-ac graph. By Theorem~\ref{chain_4-ac}, $G$ is a chain graph with multiple links such that all interior links are arcs. If $G$ is not a double ray, then we may suppose that $G$ has a block decomposition $\set{L_n}:{n \in J}$ where $J$ is an interval in $\Set{0} \cup \N$ containing $0$. We show that $L_0$ is either an arc or circle, for then any end-link of the block decomposition of $G$ is a circle or an arc, and all interior links are arcs -- and the theorem follows immediately. 

\smallskip

\textbf{Claim:} \emph{$L_0$ is either a circle or an arc.} 
Indeed, let $v_0$ be the linking vertex $L_0 \cap L_1$. By Theorem~\ref{chain_4-ac}~(\ref{chain_4-ac(3)}), if $L_0$ is not an arc, $v_0$ has degree $2$ in $L_0$. Let $e_0$ and $e_1$ be the two edges of $L_0$ incident with $v_0$. If $L_0$ is not a circle, we may suppose without loss of generality that $e_0 = v_0w$ where $w$ has degree $3$ in $L_0$. Write $e_2,e_3$ for the other two edges incident with $w$. Pick four points $x_0,\ldots,x_3$ with $x_i$ on the interior of $e_i$ for $0 \leq i \leq 3$ and pick $x_4$ on the interior of an edge $e_4$ in $L_1$ incident with $v_0$. Then these five points witness that $G$ is not $5$-ac: By Lemma~\ref{lem_oddevencuts}, any arc would need to start and end on the same side of the edge cut $\Set{e_0,e_1} = E(L_0 \setminus \singleton{v_0}, G \setminus L_0)$, but also needs to start and end in a neighbourhood of $w$ and a neighbourhood of $v_0$ respectively by Lemma~\ref{tripodlemma}, a contradiction.
\end{proof}

Thus, we may concentrate on cyclically connected graphs. Here, we have the following characterization.

\begin{theorem}
\label{thm_5acChar}
A cyclically connected graph $G$ is $5$-ac if and only if
\begin{enumerate}
\item $G$ has maximum degree $4$,
\item no $3$, or fewer, vertices of $G$ cut $G$ into $5$ or more components,
\item $G$ is not a cycle graph of three (non-trivial) links $L_0,L_1,L_2$ such that the linking vertex $v \in L_0 \cap L_1$ has both $\operatorname{deg}_{L_0}(v) = 2 = \operatorname{deg}_{L_1}(v)$, and
\item $G$ is not the union of three (edge-disjoint) connected subgraphs $L_0,L_1,L_2$ with two linking vertices $v,a$ such that $L_0 \cap L_1 = L_0 \cap L_2 = L_1 \cap L_2 = \Set{v,a}$ and $\operatorname{deg}_{L_2}(v) = 2$.
\end{enumerate}
\end{theorem}

Note that the combinatorial condition (2) is equivalent to the topological statement (2${}'$) 
`no $3$ points of $|G|$ cut $|G|$ into $5$ or more components', and this is what we use below. (To see this equivalence, observe that (2${}'$) is automatically stronger than (2), and for the converse, replace any point of $|G|$ in the interior of an edge with one of the vertices at the ends of the edge.)
We also remark that the three graphs sketched below witness that even given (1), conditions (2)--(4) are mutually independent. A $K_5^-$, i.e. a $K_5$ with one of the edges removed, violates (2) but satisfies (3) and (4). %Check: Deleting any two vertices does leave at most $2$ components. 
Similarly, the second graph is a non-$5$-ac graph which fails (3) (as the diagram shows, it is a cycle graph of the type excluded by (3)) but satisfies (2) and (4).  
Finally, the third graph below satisfies (2) and (3) but not condition (4). 
(To see that (4) is violated, consider the decomposition as shown in the diagram, where the restriction imposed by (4) on degrees fails.)

\vspace{8pt}

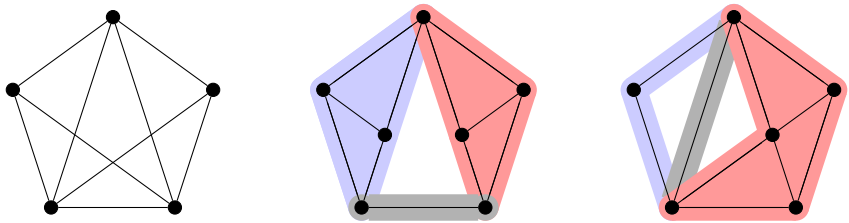
\begin{minipage}{.32\textwidth}
\begin{center}
\begin{tikzpicture}[scale=.7]
\node[circle, draw, inner sep=0.6pt, minimum width=5pt, fill=black]  (a) at (90:2cm) {};
\node[circle, draw, inner sep=0.6pt, minimum width=5pt, fill=black]  (b) at (162:2cm) {};
\node[circle, draw, inner sep=0.6pt, minimum width=5pt, fill=black]  (c) at (234:2cm) {};
\node[circle, draw, inner sep=0.6pt, minimum width=5pt, fill=black]  (d) at (306:2cm) {};
\node[circle, draw, inner sep=0.6pt, minimum width=5pt, fill=black]  (e) at (18:2cm) {};

%\draw (a) -- (b) node [circle, red, draw, inner sep=0.6pt, minimum width=5pt, fill=red, midway] {$$};
%\draw (a) -- (c) node [circle, red, draw, inner sep=0.6pt, minimum width=5pt, fill=red, midway] {$$};
%\draw (a) -- (d) node [circle, red, draw, inner sep=0.6pt, minimum width=5pt, fill=red, midway] {$$};
%\draw (a) -- (e) node [circle, red, draw, inner sep=0.6pt, minimum width=5pt, fill=red, midway] {$$};
\draw (a) -- (b);
\draw (a) -- (c);
\draw (a) -- (d);
\draw (a) -- (e);

\draw (b) -- (c);
\draw (b) -- (d);

%\draw (c) -- (d) node [circle, red, draw, inner sep=0.6pt, minimum width=5pt, fill=red, midway] {$$};
\draw (c) -- (d);
\draw (c) -- (e);

\draw (d) -- (e);
\end{tikzpicture}
\end{center}
\end{minipage}
\begin{minipage}{.32\textwidth}
\begin{center}
%\begin{tikzpicture}[scale=.7, use Hobby shortcut]
%\begin{scope}[xshift=-1.5 cm]
%\draw (0,0) circle (1.5cm);
%
%\node[circle, draw, inner sep=0.6pt, minimum width=4pt, fill=black]  (b1) at (-90:1.5cm) {};
%\node[circle, draw, inner sep=0.6pt, minimum width=4pt, fill=black]  (c1) at (-60:1.5cm) {};
%\end{scope}
%
%\begin{scope}[xshift=1.5 cm]
%\node[circle, draw, inner sep=0.6pt, minimum width=4pt, fill=black]  (a) at (-1.5,0) {};%
%
%\draw (0,0) circle (1.5cm);
%
%\node[circle, draw, inner sep=0.6pt, minimum width=4pt, fill=black]  (b2) at (-90:1.5cm) {};
%\node[circle, draw, inner sep=0.6pt, minimum width=4pt, fill=black]  (c2) at (-120:1.5cm) {};
%\end{scope}
%
%
%
%
%\node[circle, draw, inner sep=0.6pt, minimum width=4pt, fill=black]  (d1) at (-1,-2) {};
%
%\node[circle, draw, inner sep=0.6pt, minimum width=4pt, fill=black]  (d2) at (1,-2) {};
%
%\draw (d1) -- (d2);
%\draw (b1) -- (d1);
%\draw (c1) -- (d1);
%\draw (b2) -- (d2);
%\draw (c2) -- (d2);
%
%
%\end{tikzpicture}
\begin{tikzpicture}[scale=.7]
\tikzstyle{enddot}=[circle,inner sep=0cm, minimum size=10pt,color=hellrot,fill=hellrot]
\tikzstyle{markline}=[draw=hellrot,line width=10pt]
\colorlet{hellblau}{blue!20!white}
\colorlet{hellrot}{red!40!white}
\colorlet{hellgrau}{black!30!white}

\node[circle, draw, inner sep=0pt, minimum size=5pt, fill=black]  (a) at (90:2cm) {};
\node[circle, draw, inner sep=0pt, minimum size=5pt, fill=black] (b) at (162:2cm) {};
\node[circle, draw, inner sep=0pt, minimum size=5pt, fill=black]  (c) at (234:2cm) {};
\node[circle, draw, inner sep=0pt, minimum size=5pt, fill=black]  (d) at (306:2cm) {};
\node[circle, draw, inner sep=0pt, minimum size=5pt, fill=black]  (e) at (18:2cm) {};

\node[circle, draw, inner sep=0pt, minimum size=5pt, fill=black]  (f) at (-18:.77cm) {};
\node[circle, draw, inner sep=0pt, minimum size=5pt, fill=black]  (g) at (198:.77cm) {};

%\draw (a) -- (b) node [circle, red, draw, inner sep=0.6pt, minimum width=5pt, fill=red, midway] {$$};
%\draw (a) -- (c) node [circle, red, draw, inner sep=0.6pt, minimum width=5pt, fill=red, midway] {$$};
%\draw (a) -- (d) node [circle, red, draw, inner sep=0.6pt, minimum width=5pt, fill=red, midway] {$$};
%\draw (a) -- (e) node [circle, red, draw, inner sep=0.6pt, minimum width=5pt, fill=red, midway] {$$};
\draw (a) -- (b);
\draw (a) -- (c);
\draw (a) -- (d);
\draw (a) -- (e);

\draw (b) -- (c);
%\draw (b) -- (d);

%\draw (c) -- (d) node [circle, red, draw, inner sep=0.6pt, minimum width=5pt, fill=red, midway] {$$};
\draw (c) -- (d);

\draw (f) -- (e);
\draw (g) -- (b);

\draw (d) -- (e);

\begin{scope}[on background layer]
   
    \draw[hellblau,line width=10pt] (a.center) -- (b.center);
     \draw[hellblau,line width=10pt] (b.center) -- (c.center);
     \node[circle,inner sep=0cm, minimum size=10pt,color=hellblau,fill=hellblau] at (a){}; 
   \node[circle,inner sep=0cm, minimum size=10pt,color=hellblau,fill=hellblau] at (b){}; 
   \node[circle,inner sep=0cm, minimum size=10pt,color=hellblau,fill=hellblau] at (c){};

        \draw[hellblau,line width=10pt](a.center) -- (b.center); 
      \draw[hellblau,line width=10pt] (b.center) -- (c.center);
             \draw[hellblau,line width=10pt] (c.center) -- (g.center);
         \draw[hellblau,line width=10pt] (g.center) -- (a.center);
   \draw[fill=hellblau] (a.center) -- (b.center) -- (c.center) -- (g.center) -- cycle;
   \node[circle,inner sep=0cm, minimum size=10pt,color=hellblau,fill=hellblau] at (a){}; 
   \node[circle,inner sep=0cm, minimum size=10pt,color=hellblau,fill=hellblau] at (b){}; 
   \node[circle,inner sep=0cm, minimum size=10pt,color=hellblau,fill=hellblau] at (c){}; 
   \node[circle,inner sep=0cm, minimum size=10pt,color=hellblau,fill=hellblau]at (g){};  
   
     \draw[markline] (a.center) -- (e.center); 
      \draw[markline] (e.center) -- (d.center);
             \draw[markline] (d.center) -- (f.center);
         \draw[markline] (f.center) -- (a.center);
   \draw[fill=hellrot] (a.center) -- (e.center) -- (d.center) -- (f.center) -- cycle;
   \node[enddot] at (a){}; 
   \node[enddot] at (e){}; 
   \node[enddot] at (d){}; 
   \node[enddot] at (f){};  

    \draw[hellgrau,line width=10pt] (c) -- (d); 
     \node[circle,inner sep=0cm, minimum size=10pt,color=hellgrau,fill=hellgrau] at (c){}; 
   \node[circle,inner sep=0cm, minimum size=10pt,color=hellgrau,fill=hellgrau] at (d){}; 
\end{scope}

\end{tikzpicture}
\end{center}
\end{minipage}
\begin{minipage}{.32\textwidth}
\begin{center}
\begin{tikzpicture}[scale=.7]
\tikzstyle{enddot}=[circle,inner sep=0cm, minimum size=10pt,color=hellrot,fill=hellrot]
\tikzstyle{markline}=[draw=hellrot,line width=10pt]
\colorlet{hellblau}{blue!20!white}
\colorlet{hellrot}{red!40!white}
\colorlet{hellgrau}{black!30!white}

\node[circle, draw, inner sep=0pt, minimum size=5pt, fill=black]  (a) at (90:2cm) {};
\node[circle, draw, inner sep=0pt, minimum size=5pt, fill=black] (b) at (162:2cm) {};
\node[circle, draw, inner sep=0pt, minimum size=5pt, fill=black]  (c) at (234:2cm) {};
\node[circle, draw, inner sep=0pt, minimum size=5pt, fill=black]  (d) at (306:2cm) {};
\node[circle, draw, inner sep=0pt, minimum size=5pt, fill=black]  (e) at (18:2cm) {};

\node[circle, draw, inner sep=0pt, minimum size=5pt, fill=black]  (f) at (-18:.77cm) {};

%\draw (a) -- (b) node [circle, red, draw, inner sep=0.6pt, minimum width=5pt, fill=red, midway] {$$};
%\draw (a) -- (c) node [circle, red, draw, inner sep=0.6pt, minimum width=5pt, fill=red, midway] {$$};
%\draw (a) -- (d) node [circle, red, draw, inner sep=0.6pt, minimum width=5pt, fill=red, midway] {$$};
%\draw (a) -- (e) node [circle, red, draw, inner sep=0.6pt, minimum width=5pt, fill=red, midway] {$$};
\draw (a) -- (b);
\draw (a) -- (c);
\draw (a) -- (d);
\draw (a) -- (e);

\draw (b) -- (c);
%\draw (b) -- (d);

%\draw (c) -- (d) node [circle, red, draw, inner sep=0.6pt, minimum width=5pt, fill=red, midway] {$$};
\draw (c) -- (d);
\draw (c) -- (e);

\draw (d) -- (e);

\begin{scope}[on background layer]
   
    \draw[hellblau,line width=10pt] (a.center) -- (b.center);
     \draw[hellblau,line width=10pt] (b.center) -- (c.center);
     \node[circle,inner sep=0cm, minimum size=10pt,color=hellblau,fill=hellblau] at (a){}; 
   \node[circle,inner sep=0cm, minimum size=10pt,color=hellblau,fill=hellblau] at (b){}; 
   \node[circle,inner sep=0cm, minimum size=10pt,color=hellblau,fill=hellblau] at (c){}; 
   
    \draw[hellgrau,line width=10pt] (a) -- (c); 
     \node[circle,inner sep=0cm, minimum size=10pt,color=hellgrau,fill=hellblau] at (a){}; 
   \node[circle,inner sep=0cm, minimum size=10pt,color=hellgrau,fill=hellblau] at (c){}; 
   
     \draw[markline] (a.center) -- (e.center); 
      \draw[markline] (e.center) -- (d.center);
       \draw[markline] (d.center) -- (c.center);
        \draw[markline] (c.center) -- (f.center);
         \draw[markline] (f.center) -- (a.center);
   \draw[fill=hellrot] (a.center) -- (e.center) -- (d.center) -- (c.center) -- (f.center) -- cycle;
   \node[enddot] at (a){}; 
   \node[enddot] at (e){}; 
   \node[enddot] at (d){}; 
   \node[enddot] at (c){};
   \node[enddot] at (f){};  

\end{scope}

\end{tikzpicture}
\end{center}
\end{minipage}
\vspace{6pt}

We split the proof of Theorem~\ref{thm_5acChar} into two parts. First, in Proposition~\ref{prop_5acNec}, we will show that the four conditions listed in the characterization are necessary (replacing (2) with (2${}'$) where convenient). In Proposition~\ref{prop_5acSuff} further below, we will then show the converse direction.

\begin{prop}
\label{prop_5acNec}
Any cyclically connected $5$-ac graph satisfies properties (1)--(4) above.
\end{prop}

\begin{proof}
For (1), suppose that $v$ is a vertex of degree at least $5$, and $x_0, \ldots, x_4$ are chosen from the interior of distinct edges of $G$ incident with $v$. Suppose for a contradiction that there is an arc $\alpha$ in $G$ though all five points. Applying Lemma~\ref{tripodlemma} to $[v,x_0] \cup [v,x_1] \cup [v,x_2]$, we know that $v$ is an interior point of $\alpha$ and may assume that one  endpoint of $\alpha$ lies say on $(v,x_0]$. Next, applying Lemma~\ref{tripodlemma} with $[v,x_1] \cup [v,x_2] \cup [v,x_3]$ we may assume that the second endpoint of $\alpha$ lies say on $(v,x_1]$. But applying Lemma~\ref{tripodlemma} once again with $[v,x_2] \cup [v,x_3] \cup [v,x_4]$, we see that $\alpha$ is forced to have a third endpoint, a contradiction.

Condition (2) follows from Lemma~\ref{cutting_set}. 

For (3), suppose $G$ is a cycle graph of three (non-trivial) links $L_0,L_1,L_2$ with $L_0 \cap L_1 = \Set{v}$, $L_0 \cap L_2 = \Set{v_1}$ and $L_1 \cap L_2 = \Set{v_2}$ such that the linking vertex $v \in L_0 \cap L_1$ say has both $\operatorname{deg}_{L_0}(v) = 2 = \operatorname{deg}_{L_1}(v)$. Pick points $x_0,x_1$ on the interior of the distinct edges in $L_0$ incident with $v$, points $x_2,x_3$ on the interior of the distinct edges in $L_1$ incident with $v$, and $x_4$ on the interior of some edge in $L_2$. We claim that these five points witness that $G$ cannot be $5$-ac. To see this, observe first that Lemma~\ref{tripodlemma} implies that any potential arc $\alpha \colon [0,1] \to G$ containing $x_0, \ldots, x_4$ has to start and end inside $[v,x_0] \cup [v,x_1] \cup [v,x_2] \cup [v,x_3]$. In particular, $x_4$ lies on the interior of $\alpha$, and so also $v_1$ and $v_2$ lie on the interior of $\alpha$. Without loss of generality, let $ 0 < t_1 < t_2 < 1$ be the points where $\alpha(t_i) = v_1$ for $i=1,2$. Now following the arc $\alpha \restriction [0,t_1]$ backwards in time, we will first encounter say $x_0$ at time $0 \leq s_0 < t_1$. Similarly, following the arc $\alpha \restriction [t_2,1]$ forwards in time, we will first encounter say $x_2$ at time $t_2 < s_2 \leq 1$. Now, however, the points $x_0,\ldots,x_3$ are contained in the space $Y=G \setminus \alpha \restriction (s_0,s_2)$. But $v$ is a $4$-cut point of $Y$ with all $x_i$ contained in different components of $Y-v$. As in Lemma~\ref{cutting_set}, it follows that the set $\Set{x_1,\ldots,x_4}$ cannot be covered by the two disjoint arcs $\alpha \restriction [0,s_0]$ and $\alpha \restriction [s_2,1]$, a contradiction.
%\begin{center}
%\includegraphics[scale=.4]{not5ac3.png}%
%\end{center}

For (4), the argument is somewhat similar to the previous case. Pick points $x_0,x_1$ on the interior of the edges incident with $v$ in $L_0$ and $L_1$ respectively, points $x_2,x_3$ on the interior of the distinct edges in $L_2$ incident with $v$, and $x_4$ on the interior of some edge $e=ab$ in $L_2$ incident with $a$ (where, without loss of generality, we assume that $b \neq v$). We claim that these five points witness that $G$ cannot be $5$-ac. To see this, observe first that Lemma~\ref{tripodlemma} implies that any potential arc $\alpha \colon [0,1] \to G$ containing $x_0, \ldots, x_4$ has to start and end inside $[v,x_0] \cup [v,x_1] \cup [v,x_2] \cup [v,x_3]$. In particular, $x_4$ lies on the interior of $\alpha$, and so also $a$ and $b$ lie on the interior of $\alpha$, too. Without loss of generality, let $ 0 < t_1 < t_2 < 1$ be the points where $\alpha(t_i) = v_1$ for $i=1,2$. Now following the arc $\alpha \restriction [0,t_1]$ backwards in time, if $\alpha$ continues in $L_0$ or in $L_1$, we can argue similar to the previous case. If, however, $\alpha$ stays in $L_2$, then without loss of generality there are $s_2,s_3$ with $0 < s_2 < t_1 < t_2 < s_3 < 1$ with $\alpha(s_2) = x_2$ and $\alpha(s_3) = x_3$, and again we can argue that $v$ is a $4$-cut point of $Y=G \setminus \alpha \restriction (s_0,s_2)$ with all $x_i$ contained in different components of $Y-v$, and we get a contradiction as before.
\end{proof}

\begin{prop}
\label{prop_5acSuff}
Let $G$ be a cyclically connected graph such that
\begin{enumerate}
\item\label{assump1} $G$ has maximum degree $4$,
\item\label{assump2} no $3$ points of $|G|$ cut $|G|$ into $5$ or more components,
\item\label{assump3} $G$ is not a cycle graph of three (non-trivial) links $L_0,L_1,L_2$ such that the linking vertex $v \in L_0 \cap L_1$ has both $\operatorname{deg}_{L_0}(v) = 2 = \operatorname{deg}_{L_1}(v)$, and
\item\label{assump4} $G$ is not the union of three (edge-disjoint) connected subgraphs $L_0,L_1,L_2$ with two linking vertices $v,a$ such that $L_0 \cap L_1 = L_0 \cap L_2 = L_1 \cap L_2 = \Set{v,a}$ and $\operatorname{deg}_{L_2}(v) = 2$.
\end{enumerate}
Then $|G|$ is 5-ac.
\end{prop}

\begin{proof}
Consider $x_0, \ldots, x_4 \in |G|$. By Lemma~\ref{lem_wlogpointsonedges}, to show that $G$ is $5$-ac, it suffices to consider points $x_i$ which lie on the interior of edges of $G$. Applying Lemma~\ref{lem_cutting_set} and Theorem~\ref{cyclic_4ac}, we see that condition~(\ref{assump2}) implies in particular that $x_0, \ldots, x_3$ lie on a basic $4$-ac space $X$, i.e. either on a cycle, a $\theta$-curve, a baguette- or a happy-face-curve. Using $2$-connectedness, we may connect $x_4$ to $X$ via to internally disjoint paths $\alpha_1,\alpha_2$, i.e.\ paths with $\alpha_i(0) = x_4$, $\alpha_i(1) \in X$, and $\alpha \restriction [0,1) \cap X = \emptyset$. If $X$ was a cycle, then all five points lie on a $\theta$-curve, so in particular they lie on a common arc, and we are done. Thus, only the three cases remain where $X$ is a $\theta$-curve, a baguette-curve, or a happy-face-curve. We now analyze each case separately.

\textbf{Case 1. $X$ a $\theta$-curve.} 

Write $a$ and $b$ for the two degree-3-vertices of the $\theta$-curve, and $e_0,e_1,e_2$ for the three edges of the $\theta$-curve. As we may suppose that no $4$ vertices of $\Set{x_0,\ldots,x_4}$ lie on a cycle, we may label the edges of our $\theta$-curve such that $x_0 \in e_0$, $x_1 \in e_1$ and $x_2,x_3 \in e_2$. Since vertices in $G$ have degree at most $4$, the following cases for how that arcs $\alpha_1$ and $\alpha_2$ connect up to $X$ can occur:
\begin{enumerate}[label=(\roman*)]
\item\label{subcasei} $\alpha_1(1) = a$ and $\alpha_2(1) = b$,
\item\label{subcaseii} $\alpha_1(1) = a$ and $\alpha_2(1) \in e_i$, or
\item\label{subcaseiii} both $\alpha_1$ and $\alpha_2$ hit $X$ on interior points of edges.
\end{enumerate}

In \ref{subcaseiii}, either $\alpha_1(1)$ and $\alpha_2(1)$ lie on the same edge $e_i \subseteq X$, in which case $Y=X \cup \alpha_1 \cup \alpha_2$ is a baguette-curve containing $\Set{x_0,\ldots,x_4}$, or, by symmetry, we may assume that $\alpha_1(1) \in (a,x_0) \subseteq e_0$ and $\alpha_2(1) \in e_1 \cup e_2$, in which case $Y=\p{X \setminus (a,\alpha_1(1))} \cup \alpha_1 \cup \alpha_2$ is a $\theta$-curve containing $\Set{x_0,\ldots,x_4}$. In both cases, our five points $x_0,\ldots,x_5$ lie on a $5$-ac subspace, and we are done. 

Next, we claim that--similar to the proof of Theorem~\ref{cyclic_4ac}--case \ref{subcasei} reduces to case \ref{subcaseii}. Indeed, suppose that $\alpha_1(1) = a$ and $\alpha_2(1) = b$. Then $Y=X \cup \alpha_1 \cup \alpha_2$ is a graph with vertices $a$ and $b$ and four parallel edges $e_0,\ldots, e_3$ with $x_0 \in e_0$, $x_1 \in e_1$, $x_2,x_3 \in e_2$ and $x_4 \in e_3$. By assumption~(\ref{assump2}) and Lemma~\ref{lem_cutting_set}, the points $a$ and $b$ do not cut $G$ into $4$ or more components, and hence there is an arc $\delta$ in $G$ between two different edges of $Y$. By symmetry, we may assume that $\delta(0) \in (a,x_0) \subseteq e_0$. But then $Y \setminus (a,\delta(0))$ is homeomorphic to a $\theta$-curve $X'=\Set{a,b} \cup e_1 \cup e_2 \cup e_3$ with the point $x_0$ joined to $X'$ via two arcs attaching to $\delta(1)$ and $b$, i.e.\ the configuration of subcase~\ref{subcaseii}.

Thus, it remains to work through case \ref{subcaseii}. By symmetry, the following possibilities can occur:
\begin{center}
\begin{tabular}{lll}
  $\alpha_2(1) \in (a,x_0) \subseteq e_0 \subseteq X$, & 
   $\alpha_2(1) \in (x_0,b) \subseteq e_0 \subseteq X$, & 
  $\alpha_2(1) \in (a,x_2) \subseteq e_2 \subseteq X$, \\
  $\alpha_2(1) \in (x_2,x_3) \subseteq e_2 \subseteq X$,   & \multicolumn{2}{l}{or \,  $\alpha_2(1) \in (x_3,b) \subseteq e_2 \subseteq X$.}  
\end{tabular}
\end{center}

\noindent \begin{minipage}{0.7\textwidth}
Write $Y = X \cup \alpha_1 \cup \alpha_2$. In the first and third case, $Y \setminus (a,\alpha_2(1))$ is a $\theta$-curve containing $\Set{x_0,\ldots,x_4}$, and in the second and fourth case, $Y \setminus (\alpha_2(1),b)$ is a figure-8-curve containing $\Set{x_0,\ldots,x_4}$. Thus, in the first four cases, our five points $x_0,\ldots,x_5$ lie on a common $\omega$-ac subspace, and we are done. In the fifth case, we see that $Y$ is a happy-face curve, i.e.\ a cycle graph consisting of two cycles and an arc with a point $x_i$ on every single edge. For convenience, let us relabel all edges and points of $Y$ as in the  picture.
\end{minipage}
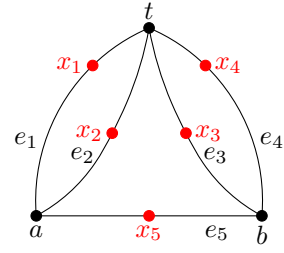
\begin{minipage}{0.3\textwidth}
\centering
\begin{tikzpicture}[scale=1, use Hobby shortcut]

\coordinate [label=below:$a$] (a) at (-1.5, -1);
\node[circle, black, draw, inner sep=0.6pt, minimum width=4pt, fill=black] at (a) {};
\coordinate [label=below:$b$] (b) at (1.5, -1);
\node[circle, black, draw, inner sep=0.6pt, minimum width=4pt, fill=black] at (b) {};
\coordinate [label=above:$t$] (t) at (0,1.5);
\node[circle, black, draw, inner sep=0.6pt, minimum width=4pt, fill=black] at (t) {};

\coordinate [label={[red]left:$x_1$}] (x1) at (-0.75,1);
\coordinate (c1) at (-.3,0.5);
\coordinate [label={[red]left:$x_2$}] (x2) at (-.49,0.1);
\coordinate (c2) at (-.72,-0.3);
\coordinate (d1) at (.3,0.5);
\coordinate [label={[red]right:$x_3$}] (x3) at (.49,0.1);
\coordinate (d2) at (.72,-0.3);
\coordinate [label={[red]right:$x_4$}] (x4) at (0.75, 1);
\coordinate [label={[red]below:$x_5$}] (x5) at (0,-1);

\coordinate [label=left:$e_1$] (e1) at (-1.35,0);
\coordinate [label=left:$e_2$] (e2) at (-0.6,-.2);
\coordinate [label=right:$e_3$] (e3) at (0.6,-.2);
\coordinate [label=right:$e_4$] (e4) at (1.35,0);
\coordinate [label=below:$e_5$] (e5) at (0.9,-1);

\draw (a) .. (x1) .. (t);
\draw (a) .. (c2) .. (x2) .. (c1) .. (t);
\draw (b) .. (d2) .. (x3) .. (d1) .. (t);
\draw (b) .. (x4) .. (t);
\draw (a) .. (x5) .. (b);

\node[circle, red, draw, inner sep=0.6pt, minimum width=4pt, fill=red] at (x1) {};
\node[circle, red, draw, inner sep=0.6pt, minimum width=4pt, fill=red] at (x2) {};
\node[circle, red, draw, inner sep=0.6pt, minimum width=4pt, fill=red] at (x3) {};
\node[circle, red, draw, inner sep=0.6pt, minimum width=4pt, fill=red] at (x4) {};
\node[circle, red, draw, inner sep=0.6pt, minimum width=4pt, fill=red] at (x5) {};
%dummy
\node[circle, white, draw, inner sep=0.6pt, minimum width=4pt, fill=white,yshift=-20pt] at (x5) {};
\end{tikzpicture}
\end{minipage}

By condition~~\ref{assump2}, the three points $a,b,t$ do not disconnect $G$ into $5$ or more components, and therefore there is an arc $\delta$ internally disjoint from $Y$ connecting some pair of edges $e_i$ and $e_j$ for $i \neq j$. Again we differentiate several subcases (up to symmetry) depending on the attaching points of $\delta$.
\begin{multicols}{2}
\begin{enumerate}[label=(\alph*)]
\item\label{case(a)} $\delta(0) \in (a,x_1)$ and $\delta(1) \in (a,x_2)$,
\item\label{case(b)} $\delta(0) \in (a,x_1)$ and $\delta(1) \in (x_2,t)$,
\item\label{case(c)} $\delta(0) \in (a,x_2)$ and $\delta(1) \in (x_3,t)$,
\item\label{case(d)} $\delta(0) \in (a,x_2)$ and $\delta(1) \in (b,x_3)$,
\item\label{case(e)} $\delta(0) \in (x_2,t)$ and $\delta(1) \in (x_3,t)$,
%\item $\delta(0) \in (x_2,t) \subseteq e_2$ and $\delta(1) \in (b,x_3) \subseteq e_3$, %(symmetric with 3rd bullet)
\item\label{case(f)} $\delta(0) \in (a,x_2)$ and $\delta(1) \in (a,x_5)$,
\item\label{case(g)} $\delta(0) \in (a,x_2)$ and $\delta(1) \in (x_5,b)$,
\item\label{case(h)} $\delta(0) \in (x_2,t)$ and $\delta(1) \in (a,x_5)$,
\item\label{case(i)} $\delta(0) \in (x_2,t)$ and $\delta(1) \in (x_5,b)$.
\end{enumerate}
\end{multicols}
Note in case \ref{case(a)}, for example, we additionally know that $(a,x_1) \subseteq e_1$ and $(a,x_2) \subseteq e_2)$, and similarly for all the cases. 
Write $Z = Y \cup \delta$. Now in \ref{case(b)} and \ref{case(c)}, $\Set{x_0,\ldots,x_4}$ lie on the common $\theta$-curve $Z \setminus \p{(a,\delta(0)) \cup (\delta(1),t)}$ respectively. In \ref{case(d)}, $Z \setminus \p{(a,\delta(0)) \cup (b,\delta(1))}$ is a figure-8-curve containing $\Set{x_0,\ldots,x_4}$. In \ref{case(e)}, $Z \setminus \p{(\delta(0),t) \cup (\delta(1),t)}$ is a figure-8-curve containing $\Set{x_0,\ldots,x_4}$. In \ref{case(g)}, $\Set{x_0,\ldots,x_4}$ lie on the common figure-8-curve $Z \setminus \p{(a,\delta(0)) \cup (\delta(1),b)}$. In \ref{case(h)}, $\Set{x_0,\ldots,x_4}$ lie on the common $\theta$-curve $Z \setminus \p{(\delta(0),t) \cup (a,\delta(1))}$. And in \ref{case(i)}, $Z \setminus \p{(\delta(0),t) \cup (\delta(1),b)}$ is a dumbbell containing $\Set{x_0,\ldots,x_4}$.

Thus, it remains to check cases \ref{case(a)} and \ref{case(f)}. Note first these cases are isomorphic (after relabeling $a:=\delta(1)$ in \ref{case(f)}, and so forth). So without loss of generality, we may assume we are in case \ref{case(a)}. Note that by assumption~\ref{assump3}, there must exist some additional arcs connecting different parts of the subgraph.
Let us work in the (connected) space $G'=|G| \setminus \singleton{t}$. First, assume there is no cut-vertex separating $A=e_1 \cup \{a\} \cup e_2$ from $B=e_3 \cup \{b\} \cup e_4$ in $G'$ (in particular, we assume $b$ is not such a cut-vertex). Then by Menger, there exists an $A-B$ walk $\beta$ in $G-t$ which is alternating with respect to $\closure{e_5}$ such that $a,b \notin \beta$.

\textbf{Claim 1:} \emph{If $\beta \cap (x_5,b) \neq \emptyset$, then we are done.} 

To see the claim, let $t \in (0,1)$ be  minimal such that $\beta(t) \in e_5$. Since we have excluded %\ref{case(c)}, \ref{case(d)} and \ref{case(e)} above, we may assume that $\beta(t) \in e_5$, and since we have excluded 
\ref{case(h)} and \ref{case(i)} above, we may assume that $\beta(0) \in (a,x_2) \subseteq e_2$. Next, let $t' \in (0,1)$ be minimal such that $x=\beta(t') \in (x_5,b)$, and consider the arc $\beta ' = \beta \restriction [0,t']$. Then $\beta'$ is an $A-x$ walk disjoint from $B$ and alternating with respect to $[a,x] \subseteq \closure{e_5}$ such that $x_5 \notin \beta'$ (this follows from the definition of `alternating'), and so we find two independent $A-x$-paths $\gamma_1$ and $\gamma_2$ in the symmetric difference of $\beta'$ and $[a,x]$ with starting vertices $a$ and $\beta(0)$ respectively such that $x_5 \in \gamma_i$ for precisely one $i$, see Lemma~\ref{lem_existencealternatingwalk}. But then $[Y \setminus \p{(a,\beta(0)) \cup e_5 }] \cup \gamma_1 \cup \gamma_2$ 
is a figure-8-curve containing $\Set{x_0,\ldots,x_4}$. 

\textbf{Claim 2:} \emph{If $\beta \cap (x_5,b) =\emptyset$, then we are also done.} 

To see this, note that $\beta \cap (x_5,b) =\emptyset$ implies, by the definition of alternating, that $x_5 \notin \beta$. Therefore, by taking the symmetric sum of $\closure{e_5}$ and $\beta$, we obtain two disjoint $A-x$-paths $\gamma_1$ and $\gamma_2$ with starting vertices $a$ and $\beta(0)$ respectively such that $x \in \gamma_i$ for precisely one $i$ (if we choose the walk $\beta$ according to the moreover-part of Lemma~\ref{lem_existencealternatingwalk}). Next, since we have excluded cases \ref{case(c)}, \ref{case(d)} and \ref{case(e)} above, we may assume that $\beta \cap (a,x_5) \neq \emptyset$, and since we have excluded cases \ref{case(h)} and \ref{case(i)} above, we may further assume that $\beta(0) \in (a,x_2) \subseteq e_2$. Thus, up to symmetry, the following four arrangements can occur:
\begin{multicols}{2}
\begin{enumerate}
\item $\gamma_1(1) = b$, $\gamma_2(1) \in (b,x_3) \subseteq e_3$,
\item $\gamma_1(1) = b$,  $\gamma_2(1) \in (x_3,t) \subseteq e_3$,
\item $\gamma_2(1) = b$, $\gamma_1(1) \in (b,x_3) \subseteq e_3$, 
\item $\gamma_2(1) = b$,  $\gamma_1(1) \in (x_3,t) \subseteq e_3$.
\end{enumerate}
\end{multicols}
In the first case, $\p{Y \setminus [e_5 \cup (a,\gamma_2(0)) \cup (b,\gamma_2(1))]} \cup \gamma_1 \cup \gamma_2$ is a figure-8-curve containing $\Set{x_0,\ldots,x_4}$. In the second case,  $\p{Y \setminus [e_5 \cup (a,\gamma_2(0)) \cup (\gamma_2(1),t)]} \cup \gamma_1 \cup \gamma_2$ is a $\theta$-curve containing $\Set{x_0,\ldots,x_4}$. In the third case,  $\p{Y \setminus [e_5 \cup (a,\gamma_2(0)) \cup (b,\gamma_1(1))]} \cup \gamma_1 \cup \gamma_2$ is a figure-8-curve containing $\Set{x_0,\ldots,x_4}$. And in the last case, the subgraph $\p{Y \setminus [e_5 \cup (a,\gamma_2(0)) \cup (\gamma_1(1),t)]} \cup \gamma_1 \cup \gamma_2$ is a $\theta$-curve containing $\Set{x_0,\ldots,x_4}$.

This completes the case checks for when there was no cut-vertex between $A$ and $B$ in $G'$. So now, we may assume that some vertex $v \in e_5 \cup \singleton{b}$ is a cut-vertex of $G'$. Without loss of generality, $v$ is chosen as close to $a$ on $\closure{e_5}$ as possible.

\textbf{Claim 3:} \emph{If $v \in (a,x_5)$, then we are done.} 

Indeed, the existence of a further cut point $v'$ separating $x_5$ from $B$ in $G-\{t,v\}$ would contradict condition~(\ref{assump3}). Therefore, by Menger (cf.\ Corollary~\ref{thm_existencealternatingwalk}), there exists an $x_5-B$ walk $\beta$ in $G-\{t,v\}$ which is alternating with respect to $[x_5,b] \subseteq e_5$. Write $z$ for the endpoint of $\beta$ on $e_3 \cup e_4$. By the excluded cases \ref{case(h)} and \ref{case(i)} above, we may assume that $\beta(1) \in (b,x_3) \subseteq e_3$. By taking the symmetric difference of $[x_5,b]$ and $\beta$, we see as in Claim 1 above that our set $\Set{x_0,\ldots,x_4}$ lies on a figure-8-curve. 

\textbf{Claim 4:} \emph{If $v \in (x_5,b)$, then then we are also done.}  
This case follows as in Claim 1 (using the fact that $v$ was chosen left-most).

\textbf{Claim 5:} \emph{Can deal with the case $v =b $.} 

Again, since $v$ was chosen left-most, our paths $\beta$ and $\closure{e_5}$ witness that $b$ has degree $2$ in $L_2:=G \setminus \p{e_3 \cup e_4}$. %Let $e = bw$ be one of the edges of $G$ contributing to this degree $2$. 
Now at this point, $\Set{t,b,w}$ would give rise to a decomposition of $G$ into 
\[L_0 = \closure{e_3}, \; L_1 = \closure{e_4}, \; \text{ and } \; L_2 = G \setminus \p{e_3 \cup e_4},\]
contradicting condition~(\ref{assump4}). Therefore, since $v = b$ was assumed to be a cut-vertex, we are forced to conclude there there must be an additional arc $\delta'$ between $e_3$ and $e_4$. As we have excluded \ref{case(b)} and \ref{case(c)} above, we may assume that $\delta'(0) \in (b,x_3) \subseteq e_3$ and $\delta'(1) \in (b,x_4) \subseteq e_4$. 

Finally, since we have dealt with Claim 1 already, we may assume that $x_5 \notin \beta$. Taking the symmetric difference of $\beta$ and $\closure{e_5}$ gives us two disjoint $A-b$ paths $\gamma_1$ and $\gamma_2$ such that $x \in \gamma_i$ for precisely one $i$ (again assuming that we choose the walk $\beta$ according to the moreover part of Lemma~\ref{lem_existencealternatingwalk}), with say $\gamma_1(0) = a$ and $\gamma_2(0) \in (a,x_2)$. But then $[Y \setminus \p{e_5 \cup (a,\gamma_2(0)) \cup (b,\delta'(0))\cup (b,\delta'(1))}] \cup \gamma_1 \cup \gamma_2$ is a figure-8-curve containing $\Set{x_0,\ldots,x_4}$, and we are done.

\textbf{Case 2. $X$ a baguette-curve.}

This case is fairly easy in comparison. If $X$ is a baguette curve with cycles $C_1,C_2$, degree-3-vertices $a,b \in C_1$ and $c,d \in C_2$ and edges $\Set{e_0,e_1}= E(C_1)$, $\Set{e_2,e_3} = E( C_2)$ and $e_4=ac,e_5=bd$ between $C_1$ and $C_2$, we may assume that $x_i \in e_i$ for $0 \leq i \leq 3$, as otherwise we are back in the $\theta$-curve case. Now consider where the arcs $\alpha_1$ and $\alpha_2$ attaching $x_4$ hit $X$. Note first that if say $\alpha_i(1) \in C_i$, then $\Set{x_0,\ldots,x_4}$ lie on a common dumbbell, and we are done. Thus, up to symmetry, the following cases remain:
\begin{multicols}{2}
\begin{enumerate}[label=(\alph*)]
\item\label{case0000} $\alpha_1(1) \in e_3 \cup e_4$,
\item\label{blaA} $\alpha_1(1) \in e_4$, $\alpha_2(1) = c$,
%\item\label{blaB} $\alpha_1(1) \in e_4$, $\alpha_2(1) =d $,
\item\label{blaC} $\alpha_1(1) \in e_4$, $\alpha_2(1) \in e_5 \cup \singleton{d}$, or
\item\label{blaD} $\alpha_1(1) =c$, $\alpha_2(1) =d$.
\end{enumerate}
\end{multicols}
In case~\ref{case0000}, we may assume by symmetry that $\alpha_1(1) \in (c,x_3)$. Then $\p{X \cup \alpha_1} \setminus \p{(c,\alpha_1(1) \cup e_5}$ is a lollipop containing $\Set{x_0,\ldots,x_4}$. Next, let $Y = X \cup \alpha_1 \cup \alpha_2$. In case~\ref{blaA}, $Y \setminus (\alpha_1(1),\alpha_2(1))$ is a baguette-curve containing $\Set{x_0,\ldots,x_4}$. In case~\ref{blaC}, $Y \setminus (\alpha_1(1),c) \cup (b,\alpha_2(1))$ is a lollipop containing $\Set{x_0,\ldots,x_4}$. Thus, it remains to analyze case~\ref{blaD} more closely. In this case, the subgraph $Y=X \cup \alpha_1 \cup \alpha_2$ is the baguette curve $X$ with an extra edge $e_6$ with endpoints $c$ and $d$. In particular, removing the vertices $c$ and $d$ from $|Y|$ would leave $4$ connected components.

Thus, using the condition that no two vertices split $|G|$ into $4$ different components (by condition~(\ref{assump2}) and Lemma~\ref{lem_cutting_set}), we know that there must be a further arc $\delta$ internally disjoint from $Y$ and connecting different components of $|Y| \setminus \{a,b\}$. Up to symmetry (as there is no structural difference between $e_2,e_3$ and $e_6$), we may assume that $\delta(0) \in (x_2,d) \subseteq e_2$. Then for the other endpoint of $\delta$, the following cases can occur:
\begin{multicols}{3}
\centering
\begin{enumerate}[label=(\roman*)]
\item\label{baguette3} $\delta(1) \in e_4$,

\item\label{baguette4} $\delta(1) \in e_5$,
\item\label{baguette5} $\delta(1) =a$,
\item\label{baguette6} $\delta(1) =b$,

\item\label{baguette1} $\delta(1) \in (x_3,d) \subseteq e_3$,
\item\label{baguette2} $\delta(1) \in (c,x_3) \subseteq e_3$,
\item\label{baguette7} $\delta(1) \in (a,x_0) \subseteq e_0$, 
\item\label{baguette8} $\delta(1) \in (x_0,b) \subseteq e_0$.
\end{enumerate}
\end{multicols}
In all cases, it is straightforward to see to verify that our five points $x_0, \ldots, x_4$ lie on a common dumbbell. %see next picture. 
%\begin{center}
%\includegraphics[scale=.4]{CaseChecks5ac2.png}%
%\end{center}
This completes the proof of Case 2.

\textbf{Case 3. $X$ a happy-face-curve.}

Again, this case is fairly easy in comparison. If $X$ is a happy face curve with cycles $C_1,C_2$, degree-4 vertex $a \in C_1 \cap C_2$, degree-3-vertices $b \in C_1 \setminus C_2$ and $c \in C_2 \setminus C_1$ and edges $\Set{e_0,e_1} =  E(C_1)$, $\Set{e_2,e_3} = E(C_2)$ and $e_4=bc$, we may assume that $x_i \in e_i$ for $0 \leq i \leq 3$, as otherwise we are back in the $\theta$-curve case. Now consider where the arcs $\alpha_1$ and $\alpha_2$ attaching $x_4$ hit $X$. Note that the $\alpha_i$ cannot hit on any $x_j$ (as they were chosen to lie on the interior of edges of $G$), nor on the center vertex $a$, by condition~(\ref{assump1}).

If $\alpha_1$ and $\alpha_2$ hit the same segment of $C_i \setminus \{a,x_j,x_k\}$, then ignoring the edge $e_4$, we see that all our 5 points lie on a figure-8-curve.

Next, if $\alpha_1$ hits $C_1$ say, and $\alpha_2$ doesn't, then it's easy to see that we are back in the discussion as in Case 1, where all our five points lie on the different edges of a happy face curve, so we are done, as we have solved this arrangement already.

Lastly, we assume that $\alpha_1$ and $\alpha_2$ hit different segments of say $C_1 \setminus \{a,x_0,x_1\}$. Let us view $C_1$ as a cycle $aex_0fx_1ga$ with vertices $a,x_0,x_1$ and three edges. After removing the edge $e_5$, we see that up to symmetry, the following three cases can occur: 
(i) $\alpha_1(1) \in e$ and $\alpha_2(1) \in f$, \
(ii) $\alpha_1(1) \in e$ and $\alpha_2(1) \in g$, or \ 
(iii) $\alpha_1(1) \in f$ and $\alpha_2(1) \in g$. 
In all three cases, we see that $X \setminus \p{e_5 \cup (a,\alpha_1(1))}$ is a dumbbell containing our five points $x_0, \ldots, x_4$.
%See the following picture:
%\begin{center}
%\includegraphics[scale=.4]{pic5ac2.png}%
%\includegraphics[scale=.4]{finalcase.png}
%\end{center}
This completes the proof.
\end{proof}

\subsection{Characterizing \texorpdfstring{$6$}{6}-ac graphs}

Our characterization of $6$-ac graphs, the main result of this section, is as follows. %Again, by Theorem~\ref{chain_5-ac}, we may limit our attention to cyclically connected graphs.

\begin{theorem} 
\label{thm_6acChar}
A graph $G$ is $6$-ac if and only if either $G$ is one of the nine $7$-ac graphs of Theorem~\ref{7ac1complex} or, after suppressing all degree-2-vertices, the graph $G$ is $3$-regular, $3$-connected, and removing any $6$ edges does not disconnect $G$ into 4 or more components.\footnote{Equivalently, if $G$ is $3$-regular, $3$-connected and not an \emph{inflated $K_4$}: there is no partition of $V(G)$ into four non-empty subsets $V_1,\ldots,V_4$ such that each $G_i=G[V_i]$ is connected and there is precisely one $G_i-G_j$ edge in $G$ for every pair $i \neq j$.}
\end{theorem}

Note that the last condition in particular implies that $G$ must be triangle-free. However, the stronger condition we chose is necessary for the characterization, as demonstrated by the following $3$-regular $3$-connected, triangle-free graphs, which both fail to be $6$-ac (in both cases consider the six points labeled 
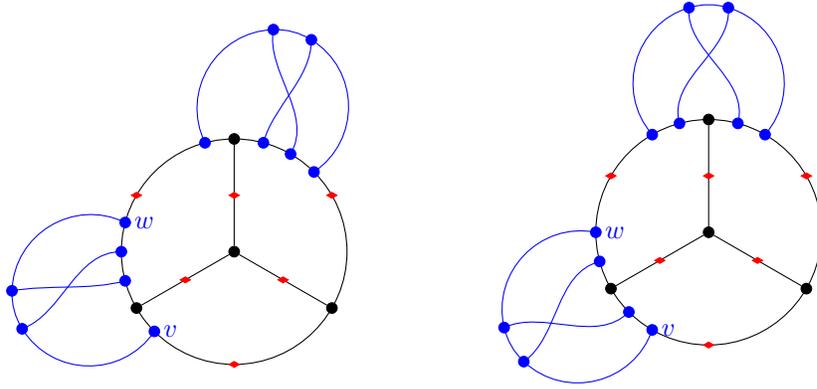
\tikz[baseline=-0.75ex]{\node[diamond, draw, red, inner sep=0.6pt, minimum width=4pt, fill=red]  {};}).

\begin{minipage}{.49\textwidth}
\begin{center} 
\begin{tikzpicture}[scale=1, use Hobby shortcut]
\node[circle, draw, inner sep=0.6pt, minimum width=4pt, fill=black]  (a) at (0,0) {};

\node[circle, draw, inner sep=0.6pt, minimum width=4pt, fill=black]  (b) at (90:1.5cm) {};

\node[circle, draw, inner sep=0.6pt, minimum width=4pt, fill=black]  (c) at (210:1.5cm) {};

\node[circle, draw, inner sep=0.6pt, minimum width=4pt, fill=black]  (d) at (330:1.5cm) {};

\draw (0,0) circle (1.5cm);
\draw (a) -- (b);
\draw (a) -- (c);
\draw (a) -- (d);
\node[diamond, red, draw, inner sep=0.6pt, minimum width=4pt, fill=red]  (x1) at (-210:1.5cm) {};

\node[diamond, red, draw, inner sep=0.6pt, minimum width=4pt, fill=red]  (x3) at (-330:1.5cm) {};

\node[diamond, red, draw, inner sep=0.6pt, minimum width=4pt, fill=red]  (x2) at (-90:1.5cm) {};

\node[diamond, draw,red, inner sep=0.6pt, minimum width=4pt, fill=red]  (x4) at (90:.75cm) {};

\node[diamond, draw, red, inner sep=0.6pt, minimum width=4pt, fill=red]  (x5) at (210:.75cm) {};

\node[diamond, draw, red, inner sep=0.6pt, minimum width=4pt, fill=red]  (x6) at (330:.75cm) {};

\coordinate [label={[blue]right:$w$}] (p1) at (165:1.5cm);
\node[circle, blue, draw, inner sep=0.6pt, minimum width=4pt, fill=blue] at (p1) {};
\coordinate (p2) at (180:1.5cm);
\node[circle, blue, draw, inner sep=0.6pt, minimum width=4pt, fill=blue]  at (p2) {};
\coordinate (p3) at (195:1.5cm);
\node[circle, blue, draw, inner sep=0.6pt, minimum width=4pt, fill=blue]  at (p3) {};
\coordinate [label={[blue]right:$v$}] (p4) at (225:1.5cm);
\node[circle, blue, draw, inner sep=0.6pt, minimum width=4pt, fill=blue]  at (p4) {};
\coordinate (p5) at (190:3cm);
\node[circle, blue, draw, inner sep=0.6pt, minimum width=4pt, fill=blue]  at (p5) {};
\coordinate (p6) at (200:3cm);
\node[circle, blue, draw, inner sep=0.6pt, minimum width=4pt, fill=blue]  at (p6) {};

\draw[blue] (p1) .. (p5) .. (p6) .. (p4);
\draw[blue, out angle=180, in angle=20] (p2) .. (p6);
%\draw[very thick,white, out angle=195, in angle=10] (p3) .. (p5);
\draw[blue, out angle=195, in angle=10] (p3) .. (p5);

\coordinate (q1) at (45:1.5cm);
\node[circle, blue, draw, inner sep=0.6pt, minimum width=4pt, fill=blue]  at (q1) {};
\coordinate (q2) at (60:1.5cm);
\node[circle, blue, draw, inner sep=0.6pt, minimum width=4pt, fill=blue]  at (q2) {};
\coordinate (q3) at (75:1.5cm);
\node[circle, blue, draw, inner sep=0.6pt, minimum width=4pt, fill=blue]  at (q3) {};
\coordinate (q4) at (105:1.5cm);
\node[circle, blue, draw, inner sep=0.6pt, minimum width=4pt, fill=blue]  at (q4) {};
\coordinate (q5) at (70:3cm);
\node[circle, blue, draw, inner sep=0.6pt, minimum width=4pt, fill=blue]  at (q5) {};
\coordinate (q6) at (80:3cm);
\node[circle, blue, draw, inner sep=0.6pt, minimum width=4pt, fill=blue]  at (q6) {};

\draw[blue] (q1) .. (q5) .. (q6) .. (q4);
\draw[blue, out angle=60, in angle=-100] (q2) .. (q6);
\draw[blue, out angle=75, in angle=-90] (q3) .. (q5);

\end{tikzpicture}
\end{center} 
\end{minipage} 
\begin{minipage}{.49\textwidth}
\begin{center} 
\begin{tikzpicture}[scale=1, use Hobby shortcut]
\node[circle, draw, inner sep=0.6pt, minimum width=4pt, fill=black]  (a) at (0,0) {};

\node[circle, draw, inner sep=0.6pt, minimum width=4pt, fill=black]  (b) at (90:1.5cm) {};

\node[circle, draw, inner sep=0.6pt, minimum width=4pt, fill=black]  (c) at (210:1.5cm) {};

\node[circle, draw, inner sep=0.6pt, minimum width=4pt, fill=black]  (d) at (330:1.5cm) {};

\draw (0,0) circle (1.5cm);
\draw (a) -- (b);
\draw (a) -- (c);
\draw (a) -- (d);
\node[diamond, red, draw, inner sep=0.6pt, minimum width=4pt, fill=red]  (x1) at (-210:1.5cm) {};

\node[diamond, red, draw, inner sep=0.6pt, minimum width=4pt, fill=red]  (x3) at (-330:1.5cm) {};

\node[diamond, red, draw, inner sep=0.6pt, minimum width=4pt, fill=red]  (x2) at (-90:1.5cm) {};

\node[diamond, draw,red, inner sep=0.6pt, minimum width=4pt, fill=red]  (x4) at (90:.75cm) {};

\node[diamond, draw, red, inner sep=0.6pt, minimum width=4pt, fill=red]  (x5) at (210:.75cm) {};

\node[diamond, draw, red, inner sep=0.6pt, minimum width=4pt, fill=red]  (x6) at (330:.75cm) {};

\coordinate [label={[blue]right:$w$}] (p1) at (180:1.5cm);
\node[circle, blue, draw, inner sep=0.6pt, minimum width=4pt, fill=blue] at (p1) {};
\coordinate (p2) at (195:1.5cm);
\node[circle, blue, draw, inner sep=0.6pt, minimum width=4pt, fill=blue]  at (p2) {};
\coordinate (p3) at (225:1.5cm);
\node[circle, blue, draw, inner sep=0.6pt, minimum width=4pt, fill=blue]  at (p3) {};
\coordinate [label={[blue]right:$v$}] (p4) at (240:1.5cm);
\node[circle, blue, draw, inner sep=0.6pt, minimum width=4pt, fill=blue]  at (p4) {};
\coordinate (p5) at (205:3cm);
\node[circle, blue, draw, inner sep=0.6pt, minimum width=4pt, fill=blue]  at (p5) {};
\coordinate (p6) at (215:3cm);
\node[circle, blue, draw, inner sep=0.6pt, minimum width=4pt, fill=blue]  at (p6) {};

\draw[blue] (p1) .. (p5) .. (p6) .. (p4);
\draw[blue, out angle=195, in angle=35] (p2) .. (p6);
\draw[blue, out angle=225, in angle=25] (p3) .. (p5);

\coordinate (q1) at (60:1.5cm);
\node[circle, blue, draw, inner sep=0.6pt, minimum width=4pt, fill=blue]  at (q1) {};
\coordinate (q2) at (75:1.5cm);
\node[circle, blue, draw, inner sep=0.6pt, minimum width=4pt, fill=blue]  at (q2) {};
\coordinate (q3) at (105:1.5cm);
\node[circle, blue, draw, inner sep=0.6pt, minimum width=4pt, fill=blue]  at (q3) {};
\coordinate (q4) at (120:1.5cm);
\node[circle, blue, draw, inner sep=0.6pt, minimum width=4pt, fill=blue]  at (q4) {};
\coordinate (q5) at (85:3cm);
\node[circle, blue, draw, inner sep=0.6pt, minimum width=4pt, fill=blue]  at (q5) {};
\coordinate (q6) at (95:3cm);
\node[circle, blue, draw, inner sep=0.6pt, minimum width=4pt, fill=blue]  at (q6) {};

\draw[blue] (q1) .. (q5) .. (q6) .. (q4);
\draw[blue, out angle=75, in angle=-95] (q2) .. (q6);
\draw[blue, out angle=105, in angle=-85] (q3) .. (q5);
\end{tikzpicture}
\end{center} 
\end{minipage}

We split the proof of Theorem~\ref{thm_6acChar} into two parts. First, in Proposition~\ref{prop_6acNec}, we will show that the three conditions mentioned in the characterization are necessary. In Proposition~\ref{prop_6acSuff} further below, we will then show the converse direction.

\begin{prop}
\label{prop_6acNec}
Let $G$ be a $6$-ac graph which different from the nine $7$-ac graphs. Then $G$ is $3$-regular, 
$3$-connected, and removing any 6 edges does not disconnect $G$ into 4 or more components.
\end{prop}

\begin{proof}
Let $G$ be a $6$-ac graph which different from the nine $7$-ac graphs.

\emph{To see that $G$ is $3$-regular}, note that $G$ contains no vertices of degree $1$, since Theorem~\ref{chain_5-ac} implies that $G$ is cyclically connected. We suppress vertices of degree $2$. Suppose for a contradiction that $v$ is a vertex of $G$ of degree $\geq 4$. Since $G$ is cyclically connected, it follows that $G$ must have another branch point. Then one of the edges incident with $v$ have a branch point as its other endpoint, say $u$. Let this edge be $e$. Pick arcs $\alpha_1, \alpha_2, \alpha_3$ interior-disjoint from each other and $e$ with $\{v\} = \alpha_i \cap \alpha_j$, such that for each $i$, $\alpha_i \setminus \{ v\}$ contains no branch points of $G$. Also pick arcs $\beta_1, \beta_2$ interior-disjoint from the $\alpha_i$, each other and $e$, $\beta_1 \cap \beta_2 = \{ u \}$ and $\beta_i \setminus \{ u \}$ contain no branch points of $G$. Pick one point from the interior of each of $\alpha_i$, $\beta_j$ and $e$. Then, by Lemma~\ref{tripodlemma}, there is no arc going through these points. 

\emph{To see that $G$ is $3$-connected}, it suffices to show, since $G$ is $3$-regular, that it is $3$-edge connected, i.e.\ that there is no partition $V(G) = A \dot\cup B$ with $|E(A,B)| \leq 2$. Note that cyclical connectedness implies that $|E(A,B)| \geq 2$. So suppose for a contradiction that there is a $2$-edge cut $E(A,B) = \{e_1,e_2\}$. Let $e_i = a_ib_i$ with $a_i \in A$ and $b_i \in B$. Note that since $G$ is cyclically connected and $3$-regular, all four endpoints of $e_1$ and $e_2$ are distinct. In particular, $a_1$ is incident with two further edges $e_3,e_4$ which both have all their endpoints in $A$, and $b_2$ is with two further edges $e_5,e_6$ which both have all their endpoints in $B$. Pick six points $x_i \in e_i$. Since any arc $\alpha$ picking up $x_1$ and $x_2$ has to have, without loss of generality, both its endpoints on the $A$-side of $G \setminus \{x_1,x_2\}$ by Lemma~\ref{lem_oddevencuts}, it follows that it cannot pick up $x_5$ and $x_6$ without violating Lemma~\ref{tripodlemma}.

\emph{Finally, suppose deleting edges $e_1, \ldots, e_6$ from $G$ leaves components $C_1, \ldots, C_k$. We claim that $k \leq 3$.} First, observe that every edge $e_i$ is incident with at most 2 different components, and by $3$-connectedness, every component $C_i$ is incident with at least $3$ distinct edges. By double counting, it follows $k \leq 4$. 

So assume that $k=4$. Then every component must be incident with precisely $3$ of the $6$ edges. We claim that the four components and the $6$ edges are arranged like a $K_4$. For this, it suffices to show that for any two components there is only one edge incident with both components. If there were two components that share three incident edges, then $G$ would be disconnected, a contradiction. And if there are two components that share two incident edges, then the other two components must also share two further incident edges, from which we conclude that the remaining two edges form a disconnection of the $G$, contradicting once again $3$-connectedness.

Thus, the $4$ components together with the $6$ edges are arranged like a $K_4$. But then it follows from Lemma~\ref{lem_oddevencuts} that if we choose an interior point $x_i$ on each of the six edges $e_i$ for $1 \leq i \leq 6$, there is no arc $\alpha$ in the graph picking up these 6 points. Indeed, suppose that the arc $\alpha$ starts at $x_1$, traverse $x_2$ up to $x_5$ in the given order and ends and $x_6$. Write $v$ for the first vertex on $\alpha$ and assume $v \in V(C_1)$.

If $e_6$ is not incident with $C_1$, consider the cut $E(C_1, G \setminus C_1) = \Set{e_1,e_i,e_j}$ of $G$ with $1<i<j<6$. Let  $\beta := \alpha \restriction [0,\alpha^{-1}(x_j)]$ and $\gamma = \alpha \restriction [\alpha^{-1}(x_j),1]$ 
denote the subarcs of $\alpha$ from $x_1$ to $x_j$ and from $x_j$ to $x_6$ respectively. By Lemma~\ref{lem_oddevencuts}, it follows that $[x_j,w]$ with $w \in C_j \subseteq \beta$ is the final segment of $\beta$. Pick $y \in (x_j,w)$. Then $\Set{x_1,x_i,y}$ is a separation of $G$ separating $x_j$ from $x_6$, contradicting the fact that $\gamma$ is an arc in $G \setminus \Set{x_1,x_i,y}$ between these very two points. Finally, if $e_6$ is also incident with $C_1$, then say $C_2$ is incident with edges $e_i,e_j,e_\ell$ with $1<i<j<\ell<6$. Considering the arcs $\beta := \alpha \restriction [\alpha^{-1}(x_i),\alpha^{-1}(x_\ell)]$ and  $\gamma = \alpha \restriction [\alpha^{-1}(x_\ell),1]$, 
we may arrive at a similar contradiction as before.
\end{proof}

Before we start proving the converse, we need the following two lemmas. Note also that the properties $3$-connected and $3$-regular imply that our graph is simple, i.e.\ (even after suppressing all degree-2-vertices)  it contains no loops or parallel edges.

\begin{lemma}
Any four points of a $3$-regular, $3$-connected graph lie on a circle or a $\theta$-curve.
\end{lemma}

\begin{proof}
Let $G$ be a $3$-regular, $3$-connected graph. It is easy to check that $3$-regularity and $2$-connectedness imply that any $4$ points $x_1, \ldots, x_4$ of $|G|$ lie on a circle, a theta curve, or a baguette curve.

In the first two cases, we are done, so it remains to show that if our four vertices lie on a baguette curve, they also lie on a $\theta$-curve. Let $C_1$ and $C_2$ be the two cycles of the baguette curve. Note that we may assume that $x_1,x_2$ lie on $C_1$ and $x_3,x_4$ on $C_2$. Now by Menger's theorem (using $3$-connectedness of $G$ and the fact that $|V(C_i)| \geq 3)$, there are $3$ vertex disjoint paths $\alpha_1, \alpha_2, \alpha_3$ from $C_1$ to $C_2$, each meeting $C_1 \cup C_2$ only in their endpoints. Note that $C_1 \setminus \{x_1,x_2\}$ consists of two segments, so one of these segments meets both say $\alpha_1$ and $\alpha_2$. But then the cycle $C_2$ together with $\alpha_1$, then walking around $C_1$ picking up $x_1$ and $x_2$, and then following back along $\alpha_2$ gives us a $\theta$-curve containing the four points $x_1, \ldots, x_4$.
\end{proof}

\begin{lemma}
\label{lem_5ac-converse}
Any five points of a $3$-regular, $3$-connected graph lie on a circle or a $\theta$-curve.
\end{lemma}

\begin{proof}
Let $G$ be a $3$-regular, $3$-connected graph and consider five points $x_1, \ldots, x_5$ of $|G|$. If any four of them lie on a circle, then we are done. 

Thus, by the previous lemma, we may assume that $x_1,\ldots, x_4$ lie on a $\theta$-curve with edges $e_1,e_2,e_3$ and vertices $a$ and $b$. By symmetry, we may assume that $x_1,x_2 \in e_1$, $x_3 \in e_2$ and $x_4 \in e_3$. Connect the last point $x_5$ to the $\theta$-curve via two new independent arcs $\alpha_1$ and $\alpha_2$. Since $G$ is $3$-regular, the two arc $\alpha_1$ and $\alpha_2$ cannot hit the $\theta$-curve in $a$ or $b$. 
If the two arcs connect to different edges of the $\theta$-curve, then in particular either $e_2$ or $e_3$ is hit, and by deleting a suitable part of $e_2$ or $e_3$ not containing $x_3$ or $x_4$ we have found a $\theta$-curve containing $x_1,\ldots, x_5$.
Thus, we may assume that the two arcs hit the same edge $e_i$, and then we have found a baguette curve of $G$ containing all five points $x_1, \ldots, x_5$. We will show that in this case, they also lie on a $\theta$-curve. 

Let $C_1$ and $C_2$ be the two cycles of the baguette curve. Up to symmetry, the following cases can occur:
\begin{center}
\begin{tabular}{l}
(1) \ $x_4,x_5 \notin C_1 \cup C_2$, \ \ 
(2) \ $x_1,x_2,x_3 \in C_1$, $x_4\in C_2$ and $x_5 \notin C_1 \cup C_2$, \\
(3) \ $x_1,x_2 \in C_1$ and $x_3,x_4,x_5 \in C_2$, or \\
(4) \ $x_1,x_2 \in C_1$, $x_3,x_4 \in C_2$ and $x_5 \notin C_1 \cup C_2$, 
\end{tabular}
\end{center}

In case (1), if two vertices lie outside of $C_1 \cup C_2$, then it's easy to find a circle inside the baguette curve containing four of the vertices. 
In case (2), we may again find a circle inside the baguette curve picking up $x_4$, $x_5$ and two of the remaining three vertices on $C_1$. 

In case (3) we follow a strategy similar to the previous lemma. By Menger and $3$-connectedness, there are $3$ vertex disjoint paths $\alpha_1, \alpha_2, \alpha_3$ from $C_1$ to $C_2$, each meeting $C_1 \cup C_2$ only in their endpoints. Note that $C_1 \setminus \{x_1,x_2\}$ has two components, so one of these segments meets say $\alpha_1$ and $\alpha_2$. But then we can follow $\alpha_1$, then walking around $C_1$ picking up $x_1$ and $x_2$, and then following $\alpha_2$ and then walk around $C_2$ back to the endpoint of $\alpha_1$ in the correct direction so as to pick up two out of the three vertices on $C_2$. So we have found four points on a circle.

In case (4), let us denote by $\beta$ the $C_1-C_2$-edge of our baguette curve containing $x_5$. As before, by Menger and $3$-connectedness, there are three vertex disjoint $C_1-C_2$ paths $\alpha_1, \alpha_2$ and $\alpha_3$.

Subcase (4a). If it is possible to choose arcs $\alpha_1,\alpha_2,\alpha_3$ such that one of them contains $x_5$, then we do so. Assume that $\alpha_1$ contains $x_5$. If a second path say $\alpha_2$ hits $C_1 \setminus \{x_1,x_2\}$ in the same segment as $\alpha_1$, then first using $\alpha_1$, then picking up $x_1,x_2$ on $C_1$, then using $\alpha_2$, and then returning to $\alpha_1$ on $C_2$ picking up at least one more point say $x_4$ gives a circle containing four of our points, and we are done. Otherwise, by symmetry and pigeon hole principle, we may assume that $\alpha_2$ and $\alpha_3$ both hit $C_1 \setminus \{x_1,x_2\}$ as well as $C_2 \setminus \{x_3,x_4\}$ in the same segments, and so it is easy finding a circle containing $x_1, \ldots, x_4$ and we are again done.

Subcase (4b). No path system between $C_1$ and $C_2$ contains $x_5$. By construction (and the fact that we have excluded subcase 3a) there is a subarc $\beta' \subseteq \beta$ such that $x_5 \in \beta'$ and say $\beta'(0) \in \alpha_1$, $\beta'(1) \in \alpha_2$ and which is otherwise disjoint from $C_1 \cup C_2 \cup \alpha_1 \cup \alpha_2 \cup \alpha_3$. Now if say $\alpha_2$ hits $C_1 \setminus \{x_1,x_2\}$ in the same segment as $\alpha_3$, then by following $\alpha_3$, picking up $x_1,x_2$ on $C_1$, then following along $\alpha_2$ until we can turn into $\beta'$ to pick up $x_5$, and then following $\alpha_1$ into $C_2$, and back to the beginning of $\alpha_3$ picking up one more point say $x_4$ on $C_2$, we have found a circle containing four of our points, and are done. Otherwise, by symmetry and pigeon hole principle, we may assume that $\alpha_1$ and $\alpha_2$ both hit $C_1 \setminus \{x_1,x_2\}$ as well as $C_2 \setminus \{x_3,x_4\}$ in the same segments, and so it is easy finding a circle containing $x_1, \ldots, x_4$ and we are again done.
\end{proof}

We are now ready to prove the converse direction of our main characterization theorem.

\begin{prop}
\label{prop_6acSuff}
Let $G$ be a simple $3$-regular, $3$-connected graph such that removing any 6 edges does not disconnect $G$ into 4 or more components. Then $G$ is $6$-ac.
\end{prop}

\begin{proof}
Pick six points $x_1, \ldots, x_6$ from $G$ which we may assume, by Lemma~\ref{lem_wlogpointsonedges}, to be interior points of edges. By Lemma~\ref{lem_5ac-converse}, there is a $\theta$-curve $\Theta$ containing the first five points $x_1, \ldots, x_5$. Write $e,f,g$ for the edges of $\Theta$ and $a,b,$ for the vertices of $\Theta$. We may assume that every edge of $e,f,g$ is incident with a point $x_i$, and so up to symmetry there are two cases to consider, namely
\begin{enumerate}[label=(\Alph*)]
\item $x_1 < x_2 < x_3 \in e$ (ordered from $a$ to $b$), $x_4 \in f$ and $x_5 \in g$, or 
\item $x_1,x_2\in e$, $x_3,x_4 \in f$ and $x_5 \in g$. 
\end{enumerate}

We may assume that $x_6 \notin \Theta$. Pick two independent $x_6-\Theta$ arcs $\alpha_1$ and $\alpha_2$. By $3$-regularity, the arcs cannot hit $\Theta$ in $a$ or $b$. 

\smallskip

\noindent\textbf{In case (A)}, if one of the arcs hits $\Theta$ on a segment of $\Theta \setminus \{x_1,\ldots,x_5\}$ incident with $a$ or $b$, then it's easy to see that all $6$ points lie on a theta curve or on a dumbbell. Similarly, if the two arcs hit the same segment of $\Theta \setminus \{a,b,x_1,\ldots,x_5\}$ then all $6$ points lie on a theta curve. Hence, it remains to investigate the case where $\alpha_1$ hits on the segment $(x_1,x_2)\subseteq e$ and $\alpha_2$ hits on the segment $(x_2,x_3) \subseteq e$. In this situation, we have a baguette curve consisting of two cycles $C_1$ and $C_2$ and disjoint $C_1-C_2$ arcs $\beta_1$ and $\beta_2$ with $x_1,x_2 \in C_1$, $x_3,x_4 \in C_2$, $x_5 \in \beta_1$ and $x_6 \in \beta_2$ (i.e.\ one point $x_i$ on every edge of the baguette curve).

By $3$-connectedness, and the fact that $|V(C_i)| \geq 3$, there exists a $C_1-C_2$ path $\beta_3$ which is \emph{alternating} with respect to $\beta_1$ and $\beta_2$. %, see \cite[Lemma 3.3.3]{Diestel}. Here, alternating path means that whenever $\beta_3$ and $\beta_i$ agree for $i = 1,2$ then $\beta_3$ walks backwards on $\beta_i$. The importance of alternating paths lies in the fact that taking the symmetric difference of the edge sets of $\beta_1,\beta_2$ and $\beta_3$ we end up with $3$ internally disjoint $C_1-C_2$ paths $\gamma_1, \gamma_2$ and $\gamma_3$, see \cite[Lemma 3.3.2]{Diestel}.
Indeed, by Lemma~\ref{lem_existencealternatingwalk} and the fact that in a $3$-regular graph, every alternating walk is automatically a path, we may choose an alternating path $\beta_3$ such that the symmetric difference $\beta_1 \triangle \beta_2 \triangle \beta_3$ yields 3 disjoint $C_1-C_2$ paths $\gamma_1, \gamma_2$ and $\gamma_3$ traversing the shared edges with $\beta_3$ in the same order as $\beta_3$. Three subcases arise.

\smallskip

\noindent (1) If there exists a $C_1-C_2$ path containing $x_5$ and $x_6$, then our $6$ points lie on a dumbbell. So may assume that $\beta_3$ does not contain both $x_5$ and $x_6$.

\smallskip

\noindent (2) If $\beta_3$ contains none of $x_5$ and $x_6$, then $\bigcup \gamma_i$ covers both $x_5$ and $x_6$, so either, a single $\gamma_i$ contains both $x_5$ and $x_6$ and we are back in (1), or we have say $x_5 \in \gamma_1$ and $x_6 \in \gamma_2$, and the following subcases arise.

$\bullet$  If $\gamma_1,\gamma_2$ hit the same segment of $C_1 \setminus \{x_1,x_2 \}$, then find a cycle picking up $5$ of our $6$ points, and we are done, and similarly, if $\gamma_1,\gamma_2$ hit the same segment of $C_2 \setminus \{x_3,x_4 \}$.

$\bullet$ Otherwise, note that the unique segment $\delta_i$ on $C_i \setminus \gamma_3$ between the endpoints of $\alpha_1$ and $\alpha_2$ contains precisely one point $x_j$. Thus, $\alpha_1 \cup \alpha_2 \cup \delta_1 \cup \delta_2$ is a circle containing 4 of our points, and it is then easy to see that using $\alpha_3$ and suitable segments of $C_i \setminus \delta_i$, we can find a $\theta$-curve containing all $6$ of our points.
%\begin{center}
%\includegraphics[scale=.4]{CasechecksA.png}%
%\end{center}
%\end{itemize}

\smallskip

\noindent (3) In the final subcase, we may assume that $\beta_3$ covers $x_5$ but not $x_6$. Then $x_6 \in \gamma_1$ say, and note that by construction of the symmetric difference, there is an arc $\delta \ni x_5$ which is internally disjoint from $C_1 \cup C_2 \cup \bigcup \gamma_i$ and has its endpoints at interior vertices of some $\gamma_i$ and $\gamma_j$. %If $i = j \geq 2$, we are back in case (2), and for $i=j=1$ we observe that by construction, both endpoints of $\delta$ lie on the same segment of $\gamma_1 \setminus \{x_6\}$,\marginpar{MAX: This isn't true: Need to choose $\beta$ more carefully.} and we are back in case (1). 
Note that it follows from Lemma~\ref{lem_existencealternatingwalk}(2) that $i\neq j$. If $i=1$ and $j \neq 1$ then we are back in case (1). And if say $i=2$ and $j=3$, then $\gamma_1$ and say $\gamma_2$ hit the same segment of $C_1 \setminus \{x_1,x_2\}$, and so by starting with $\gamma_1$, picking up $x_1,x_2$ on $C_1$, following $\gamma_2$, switching to $\delta$, then following $\gamma_3$ to $C_2$, and move on $C_2$ back to $\gamma_1$ picking up at least one more vertex of $x_3$ and $x_4$, we have found a cycle containing $5$ of our points, so we are again done. This completes the argument for case (A).

\medskip

\noindent\begin{minipage}{0.6\textwidth}
\textbf{In case (B)}, we may use the same arguments as at the beginning of (A) to see that the only critical case is where $\alpha_1$ hits on the segment $(x_1,x_2)\subseteq e$ and $\alpha_2$ hits on the segment $(x_3,x_4) \subseteq f$. Then $x_1, \ldots, x_6$ lie on the 6 edges of a $K_4$, where we label points and edges as in the figure. 

Now consider $|G| \setminus \{x_1, \ldots, x_6\}$. By the third-listed assumption on $G$, this space has at most $3$ components, and hence there must exist an arc $\delta$ internally disjoint from $K_4$ between two vertices $v,w$ of $G$ with say
$v \in (a,x_6) \; \text{ and } \; w \notin [x_5,a] \cup [x_4,a] \cup [x_6,a]$.
\end{minipage}
\begin{minipage}{0.4\textwidth}
\centering
\begin{tikzpicture}[scale=1]
\node[circle, draw, inner sep=0.6pt, minimum width=5pt, fill=black]  (a) at (0,0) {};
\node at (.3,.2) {$a$};

\node[circle, draw, inner sep=0.6pt, minimum width=5pt, fill=black]  (b) at (90:2cm) {};
\node at (90:2.3cm) {$b$};
\node at (110:2.3cm) {$e_1$};
\node[circle, draw, inner sep=0.6pt, minimum width=5pt, fill=black]  (c) at (210:2cm) {};
\node at (210:2.3cm) {$c$};
\node at (230:2.3cm) {$e_2$};
\node[circle, draw, inner sep=0.6pt, minimum width=5pt, fill=black]  (d) at (330:2cm) {};
\node at (330:2.3cm) {$d$};
\node at (350:2.3cm) {$e_3$};
\draw (0,0) circle (2cm);
\draw (a) -- (b);
\draw (a) -- (c);
\draw (a) -- (d);
\node[circle, red, draw, inner sep=0.6pt, minimum width=5pt, fill=red]  (x1) at (-210:2cm) {};
\node at (-210:2.3cm) {$x_1$};
\node[circle, red, draw, inner sep=0.6pt, minimum width=5pt, fill=red]  (x3) at (-330:2cm) {};
\node at (-330:2.3cm) {$x_3$};
\node[circle, red, draw, inner sep=0.6pt, minimum width=5pt, fill=red]  (x2) at (-90:2cm) {};
\node at (-90:2.3cm) {$x_2$};
\node[circle, draw,red, inner sep=0.6pt, minimum width=5pt, fill=red]  (x4) at (90:1cm) {};
\node at (70:1cm) {$x_4$};
\node at (80:1.6cm) {$e_4$};
\node[circle, draw, red, inner sep=0.6pt, minimum width=5pt, fill=red]  (x5) at (210:1cm) {};
\node at (190:1cm) {$x_5$};
\node[circle, draw, red, inner sep=0.6pt, minimum width=5pt, fill=red]  (x6) at (330:1cm) {};
\node at (350:1cm) {$x_6$};

\end{tikzpicture}
\end{minipage}

By symmetry, there are five cases to consider for the position of $w$, namely \begin{center}
\begin{tabular}{lll}
(a) \ $w \in (d,x_6) \subseteq e_6$, &
(b) \ $w \in (d,x_3) \subseteq e_3$, &
(c) \  $w \in (b,x_3) \subseteq e_3$, \\
(d) \  $w \in (b,x_4) \subseteq e_4$, and & 
(e) \ $w \in (b,x_1) \subseteq e_1$. &
\end{tabular}
\end{center}
%The following figure shows that in all cases but (1), there is an arc containing our 6 points $x_1, \ldots, x_6$.
%\begin{center}
%\includegraphics[scale=.7]{Casechecks.png}
%\end{center}
By inspection, one checks that in cases (b)--(e) there is an arc contained in $K_4 \cup \delta$ which contains all our $6$ points. Thus, it remains to deal with case (a). 

Let $\beta_1, \beta_2$ be the two disjoint $\{v\}-\{a,d\}$ paths in $K_4$. Since $G$ is $3$-connected, it follows from Corollary~\ref{thm_existencealternatingwalk} that there is a further $v-\{a,d\}$ path $\beta_3$ which is alternating with respect to $\{\beta_1,\beta_2\}$ (where again we use that any alternating walk must be a path by $3$-regularity). Note that by $3$-regularity, if $h$ denotes the third edge incident with $v$, then $\beta_3$ and $\delta$ agree on $h$. Starting from $v$, let $z$ be the first vertex on $\beta_3$ which lies one $K_4 \setminus e_6$, and let $y$ be the vertex before $z$ on $\beta_3$. Note that we may assume that 
\begin{enumerate}
\item either $y \in (x_6,d) \subseteq e_6$ and $z \in (x_2,d) \cup (d,x_3)$, or
\item $y \in (a,x_6) \subseteq e_6$ and $z \in (x_5,a) \cup (a,x_4)$,
\end{enumerate}
as otherwise the arc $\beta_3$ between $y$ and $z$ witnesses (up to symmetry) that we are in one of the cases (b) -- (f). Now by the fact that $\beta_3$ has been chosen according to Corollary~\ref{thm_existencealternatingwalk}, taking the symmetric difference of $\{\beta_1,\beta_2,\beta_3 \}$ yields three internally disjoint $\{v\}-\{a,d,z\}$ paths $\gamma_1,\gamma_2,\gamma_3$.

\textbf{Claim: In case (1), there are two independent $\{x_6\}-\{d,z\}$ paths internally disjoint from $K_4 - e_6$.}
This follows from Menger's Theorem~\ref{menger} once we show that inside the subgraph $\beta_1 \cup \beta_2 \cup \beta_3$ no single point separates $x_6$ from the set $ \{ z,d \}$. So suppose for a contradiction that there is such a separating point $s$. Since $\beta_2$ is a $x_6 -d$ path, we must have $s \in (x_6,d)$. But also walking from $x_6$ along the edge $e_3$ to $v$ and then along $\beta_3$ to $z$ is a $x_6 - z$ path, it follows that $s \in (x_6,d) \cap \beta_3$, and so $s \notin \gamma_1 \cup \gamma_2 \cup \gamma_3$. But then going from $x_6$ to $v$ on $e_3$, and then taking a suitable $\gamma_i$ to $\{d,z\}$ shows that $s$ cannot have been a separator.

\textbf{Claim: In case (2), there are two independent $\{x_6\}-\{a,z\}$ paths internally disjoint from $K_4 - e_6$.} Once again this will follow from Menger's Theorem~\ref{menger} once we show that inside the subgraph $\beta_1 \cup \beta_2 \cup \beta_3$ no single point separates $x_6$ from the set $ \{ z,a \}$. So suppose for a contradiction that there is such a separating point $s$. Again, we must have $s \in (a,x_6)$. But also walking from $x_6$ along the edge $e_3$ to $w$ and then along $\delta - h$ and $\beta_3 -h$ to $z$ is a $x_6 - z$ path, it follows that $s \in (a,x_6) \cap (\beta_3 - h) $. In particular, we have $s \neq v$ and $s \notin \gamma_1 \cup \gamma_2 \cup \gamma_3$. So either $s \in (a,v)$ or $s \in (v,x_6)$. In the first case, we can walk from $x_6$ to $v$ on $e_6$, and then take a suitable $\gamma_i$ to reach $\{ a,z\}$. In the second case, we can walk from $x_6$ to $w$ on $e_6$, then to $v$ on $\delta$, and then take a suitable $\gamma_i$ to $\{a,z\}$. 

Thus, it follows that by substituting that segment $[a,z]$ or $[d,z]$ of $K_4 \setminus e_6$ with those two disjoint paths, we see that all of our 6 points lie on a $\theta$-curve.
\end{proof}

\subsubsection*{Some examples of small \texorpdfstring{$6$}{6}-ac graphs}

By checking against a list of simple, 3-regular graphs of small order\footnote{See e.g.\ \url{https://en.wikipedia.org/wiki/Table_of_simple_cubic_graphs}}, we see that the only graph on $6$ vertices satisfying our characterization is $K_{3,3}$, and that the only two graphs on $8$ vertices satisfying our assumption are $K_{3,3}$ with an extra edge connecting the midpoints of two, non-adjacent edges of $K_{3,3}$ (the so-called Wagner graph), and the 3-dimensional hypercube.

\subsection{Characterizing \texorpdfstring{$7$}{7}-ac and \texorpdfstring{$\omega$}{w}-ac Graphs}

\begin{theorem}
\label{7ac1complex} 
Let $G$ be a non-degenerate graph. Then the following are equivalent:

(a) $G$ is $7$-ac, (b) $G$ is $\omega$-ac, and (c) 
\begin{itemize}
\item $G$ is homeomorphic to one of the $6$ finite graphs which are $7$-ac, or
\item $G$ is homeomorphic to one of the finite $7$-ac graphs minus possibly some endpoints. 
\end{itemize}
\end{theorem}
\begin{proof} Since the graphs mentioned in part (c) are all $\omega$-ac, it suffices to show (a) implies (c). So suppose $G$ is $7$-ac. 
The proof of Theorem~2.12 of \cite{acpaper} shows that $G$ can have at most two vertices of degree $3$ or higher.
If all vertices have degree two, then $G$ is either homeomorphic to $(0,1)$ or $S^1$. If all vertices have degree no more than $2$, but not all are degree $2$, then $G$ is either homeomorphic to $[0,1]$ or $[0,1)$.
Otherwise, extending from the (at most two) vertices of degree at least $3$, there will be a finite family of: (finite) cycles, closed intervals (i.e.\ finite paths) or half-open intervals (i.e. one-way infinite paths). The half-open intervals give rise to the objects mentioned in the second bullet point.
\end{proof}


\begin{thebibliography}{99}

%\bibitem{euler} H. Bruhn and M. Stein, \emph{On 
%end degrees and infinite cycles in locally finite graphs}, Combinatorica \textbf{27} (2007), 269--291.

\bibitem{Diestel} R. Diestel, \emph{Graph Theory}, Springer 2016, 5th edition.

\bibitem{dirac} G.A Dirac, \emph{In abstrakten Graphen vorhandene vollst\"andige 4-Graphen und ihre Unterteilungen}, Math. Nachr. \textbf{22} (1960), 61--85.

\bibitem{EGL} Y. Egawa, R. Glas and S.C. Locke, \emph{Cycles and paths through specified vertices in $k$-connected graphs},  Journal of Combinatorial Theory, Series B, \textbf{52} (1), 1991, 20--29.

\bibitem{sacpaper} B. Espinoza, P. Gartside, M. Kovan-Bakan and A. Mamatelashvili, \emph{Strong Arcwise Connectedness}, Houston Journal of Mathematics, \textbf{43} (2), 2017, 577--610. 

\bibitem{acpaper} B. Espinoza, P. Gartside and A. Mamatelashvili, \emph{$n$-Arc Connected Spaces}, Colloquium Mathematicum \textbf{130} (2013), 221--240. 

%\bibitem{EGP} B. Espinoza, P. Gartside and M. Pitz, \emph{Graph-like Compacta: Characterizations and Eulerian Loops}, submitted. 

%\bibitem{Jaeger} F.\ Jaeger, \emph{A Note on Sub-Eulerian Graphs}, Journal of Graph Theory, \textbf{3}(1) (1979), 91--93.

%\bibitem{menger} K. Menger, \emph{Zur allgemeinen Kurventheorie}, Fund.\ Math. \textbf{18} (1927), 96--115.

%\bibitem{nobling} G.\ N\"obling, \emph{Eine Versch\"arfung des $n$-Bein Satzes}, Fund.\ Math. \textbf{18} (1931), 23--28.

\bibitem{watkinsmesner} M.E.\ Watkins and D.M.\ Mesner, \emph{Cycles and connectivity of graphs}, \emph{Can. J. Math.} \textbf{19} (1967) 1319--1328.

\end{thebibliography}
\end{document}